\documentclass[a4paper,english,oneside,12pt]{amsart}
\usepackage[latin1]{inputenc}
\usepackage[T1]{fontenc}
\usepackage{amsmath}
\usepackage{amsthm}
\usepackage{mathrsfs}
\usepackage{amssymb}
\usepackage{amsfonts}
\usepackage{hyperref}

\input{prepictex}
\input{pictex}
\input{postpictex}

\DeclareMathOperator{\spa}{span}
\DeclareMathOperator{\coker}{coker}
\DeclareMathOperator{\im}{Im}

\newcommand{\labe}{\mathcal{L}}
\newcommand{\bool}{\mathcal{B}}
\newcommand{\boolj}{\hbool_J}
\newcommand{\tbool}{\tilde{\bool}}
\newcommand{\hbool}{\bool}
\newcommand{\Aalg}[1][E,\labe,\bool]{\mathcal{A}(#1)}
\newcommand{\cha}[1]{\chi_{#1}}
\newcommand{\inv}{^{-1}}
\newcommand{\spc}{\overline{\spa}}
\newcommand{\sink}{E^0_{\text{sink}}}

\newcommand{\Oalg}{\mathcal{O}}
\newcommand{\Z}{\mathbb{Z}}
\newcommand{\N}{\mathbb{N}}
\newcommand{\Kalg}{\mathscr{K}}
\newcommand{\ideal}{\mathcal{I}}
\newcommand{\Jalg}{J_{(X(E,\labe,\bool),\varphi)}}
\newcommand{\Xalg}{\mathcal{X}}

\newcommand{\OSS}{{\mathsf{X}}}
\newcommand{\DX}{{\widetilde{\mathcal{D}}_\OSS}}
\newcommand{\alwords}{\labe^\#(E)}
\newcommand{\emptyword}{\epsilon}

\theoremstyle{plain}
\newtheorem{theorem}{Theorem}[section]
\newtheorem{lemma}[theorem]{Lemma}
\newtheorem{corollary}[theorem]{Corollary}
\newtheorem{proposition}[theorem]{Proposition}

\theoremstyle{definition}
\newtheorem{remark}[theorem]{Remark}

\newtheorem{definition}[theorem]{Definition}
\newtheorem{example}[theorem]{Example}

\begin{document}

\title{$C^*$-algebras of labelled graphs III - $K$-theory computations}
\date{\today}

\author{Teresa Bates}
\address{Teresa Bates \\
School of Mathematics and Statistics \\
The University of NSW \\
UNSW  Sydney 2052 \\ AUSTRALIA} \email{teresa@unsw.edu.au}

\author{Toke Meier Carlsen}
\address{Toke Meier Carlsen \\
Department of Mathematical Sciences \\
Norwegian University of Science and Technology (NTNU) \\
7491 Trondheim \\
NORWAY} \email{Toke.Meier.Carlsen@math.ntnu.no}

\author{David Pask}
\address{David Pask \\
School of Mathematics and Applied Statistics \\
Austin Keane Building (15) \\
University of Wollongong \\
NSW 2522 \\ AUSTRALIA} \email{dpask@uow.edu.au}

\thanks{This research was supported by the Australian Research Council, The Research Council of Norway and The Danish Natural Science Research Council.}

\begin{abstract}
In this paper we give a formula for the $K$-theory of the $C^*$-algebra of a weakly left-resolving labelled space. This is done by realising the $C^*$-algebra of a weakly left-resolving labelled space as the Cuntz-Pimsner algebra of a $C^*$-correspondence. As a corollary we get a gauge invariant uniqueness theorem for the $C^*$-algebra of any weakly left-resolving labelled space. In order to achieve this we must modify the definition of the $C^*$-algebra of a weakly left-resolving labelled space.

We also establish strong connections between the various classes of $C^*$-algebras which are associated with shift spaces and labelled graph algebras. Hence, by computing the $K$-theory of a labelled graph algebra we are providing a common framework for computing the $K$-theory of graph algebras, ultragraph algebras, Exel-Laca algebras, Matsumoto algebras and the $C^*$-algebras of Carlsen.

We provide an inductive limit approach for computing the $K$-groups of an important class of labelled graph algebras, and give examples.
\end{abstract}

\maketitle

\section{Introduction}

The purpose of this paper is to provide a formula for and a practical method of computing the $K$-theory of the $C^*$-algebra $C^*(E,\mathcal{L},\mathcal{B})$ of a weakly left-resolving labelled space $(E,\mathcal{L},\mathcal{B})$ (see Theorem \ref{theorem:ktheory} and Theorem \ref{thm:ktheorym}). To do this we first realise $C^*(E,\mathcal{L},\mathcal{B})$ as a Cuntz-Pimsner algebra (this is done in Theorem \ref{theorem:main}, see also Remark \ref{remark:nn}) and then use the results of \cite{MR2102572} to identify the $K$-groups in question as the kernel and cokernel of a certain map. Then, for an important class of labelled graph $C^*$-algebras, we give a
procedure for computing these groups through an inductive limit process (see Theorem \ref{thm:howtocompute}). Our construction of $C^*(E,\mathcal{L},\mathcal{B})$ as a Cuntz-Pimsner algebra also allows us to derive a gauge invariant uniqueness theorem for $C^*(E,\mathcal{L},\mathcal{B})$ (see Corollaries \ref{cor:giut} and \ref{cor:giutm}) and a proof that $C^*(E,\mathcal{L},\mathcal{B})$ is nuclear (see Corollary \ref{corol:iota} and Remark \ref{remark:nn}), and, under a suitable countability condition, satisfies the Universal Coefficient Theorem of \cite{MR894590} (see Corollary \ref{cor:nuclear} and Remark \ref{remark:nn}).

As was already pointed out in \cite{bp3}, the original definition of $C^*(E,\mathcal{L},\mathcal{B})$ has to be modified in order to avoid degeneracy of certain projections when $E$ has sinks. Furthermore, in \cite{BPW} it was pointed out that the Gauge Invariant Uniqueness Theorem stated in \cite{MR2304922} was incorrect and a new (and correct) gauge invariant uniqueness theorem for $C^*(E,\mathcal{L},\mathcal{B})$ when $(E,\mathcal{L},\mathcal{B})$, in addition to being weakly left-resolving, is so-called set-finite (every set of vertices in $\mathcal{B}$ emits a finite number of different labels) and $\mathcal{B}$ is closed under relative complements was proved. In order to get a gauge invariant uniqueness theorem for the $C^*$-algebra of any weakly left-resolving labelled space it is necessary to change the definition of $C^*(E,\mathcal{L},\mathcal{B})$ in the case where $\mathcal{B}$ is not closed under relative complements (see the appendix for details). Since this makes things a bit more complicated, we have chosen to only work with weakly left-resolving labelled spaces $(E,\mathcal{L},\mathcal{B})$ where $\mathcal{B}$ is closed under relative complements (we call such  labelled spaces  \emph{normal}, see Subsection \ref{sec:labelled space}) in the main part of the paper, but in the appendix we show how to deal with the general case.

The motivation for our $K$-theory computations is to investigate the relationship between certain dynamical invariants of shift spaces and the $K$-theoretical invariants of $C^*$-algebras associated to these shift spaces. This connection was first brought to light in the work of Cuntz and Krieger in \cite{ck} where they showed how to associate a $C^*$-algebra $\mathcal{O}_A$ to a finite $0$-$1$ matrix $A$ with no zero rows or columns, provided that the matrix satisfied a certain condition called (I). In \cite[Proposition 3.1]{c} it was shown that the $K$-groups of a Cuntz-Krieger algebra are isomorphic to the Bowen-Franks groups of the shift of finite type $\textsf{X}_A$ associated to $A$ (see \cite{BF}). Thus, a deep connection was established between the combinatorially-defined $C^*$-algebra $\mathcal{O}_A$ and the (one-sided) shift of finite type $\textsf{X}_A$. Several generalisations of Cuntz-Krieger algebras have now been widely studied.

Combining the universal algebra approach of \cite{aHR} and the graphical approach to Cuntz-Krieger algebras begun in \cite{ew}, graph algebras were introduced in \cite{kpr}. Graph algebras were originally defined for graphs satisfying a finiteness condition -- the need for this condition was removed by Fowler and Raeburn in \cite{fr} (see also \cite{flr}). Using a different approach, Exel and Laca showed how to associate a $C^*$-algebra to an infinite $0$-$1$ matrix with no zero rows or columns in \cite{el}. A link between graph algebras and Exel-Laca algebras was provided by the ultragraph algebras introduced by Tomforde (see \cite{tom,tom2}).

Motivated by the symbolic dynamical data contained in a Cuntz-Krieger algebra, Matsumoto provided a generalisation of Cuntz-Krieger algebras by associating to an arbitrary two-sided shift space $\Lambda$ over a finite alphabet a $C^*$-algebra $\mathcal{O}_\Lambda$ (see \cite{MR1454478, MR1646513, MR1691469, MR1688137, MR1774695, MR1764930}). Later Carlsen and Matsumoto modified this construction (see \cite{MR2091486, MR1852456}), and Carlsen further modified the definition of $\mathcal{O}_\Lambda$ and extended it to one-sided shift spaces in \cite{TMC7} (see \cite{CS} for a discussion of the relationship between the three different definitions of $\mathcal{O}_\Lambda$ and for an alternative construction of the $C^*$-algebra constructed in \cite{TMC7}). The results of \cite{MR1646513, MR1691469, m99, MR1869072, MR1852456} give connections between certain dynamical invariants of $\Lambda$ and $K$-theoretical invariants of $\mathcal{O}_\Lambda$ in the same spirit as \cite{c}.

By adapting the left-Krieger cover construction given in \cite{MR0776312}, any shift space over a countable alphabet may be presented by a left-resolving labelled graph. Hence, the labelled graph algebras introduced in \cite{MR2304922} provide a method for associating a $C^*$-algebra to a shift space over a countable alphabet. As we shall see in Section \ref{sec:examples}, the class of labelled graph algebras contains the class of graph algebras and all of the classes of $C^*$-algebras discussed above. By computing the $K$-theory of a labelled graph algebra we will therefore be providing a unified approach to computing the $K$-theory of this wide collection of $C^*$-algebras.

The paper is organised as follows: Section \ref{sec:bg} contains some background on labelled graphs, labelled spaces and their $C^*$-algebras. In Section \ref{sec:lga=cpa} we give our key result, Theorem \ref{theorem:main}, which shows that the $C^*$-algebra of a weakly left-resolving normal labelled space may be realised as a Cuntz-Pimsner algebra in the sense of \cite{MR2102572}, and we also show a gauge invariant uniqueness theorem for it (Corollary \ref{cor:giut}) and that it is nuclear (Corollary \ref{corol:iota}), and, under a suitable countability condition, satisfies the Universal Coefficient Theorem of \cite{MR894590} (Corollary \ref{cor:nuclear}). In Section \ref{sec:k-theory} we use the general results of \cite{MR2102572} to provide a formula for the $K$-theory of the $C^*$-algebra of a weakly left-resolving normal labelled space (see Theorem \ref{theorem:ktheory}). Then in Section \ref{sec:examples} we show how our $K$-theory formulas reduce to those for graph algebras, ultragraph algebras and Matsumoto algebras in the appropriate cases. Furthermore in Proposition \ref{proposition:oss} we realise the Carlsen algebra $C^* ( \OSS )$ associated to a one-sided shift space $\OSS$ as a labelled graph algebra and hence compute its $K$-theory (see Corollary~\ref{cor:k-oss}). In Section \ref{sec:howdismantleanatombomb} we show how to compute the $K$-theory of certain labelled graph algebras as inductive limits; this is done in analogy with the computations of \cite{MR1911208}, and in Section \ref{sec:computations} we provide a few examples of the computations outlined in Section \ref{sec:howdismantleanatombomb}. Finally, in the appendix we give an example that illustrates the necessity of changing the definition of the $C^*$-algebra of a weakly left-resolving labelled space which is not normal in order to obtain a gauge invariant uniqueness theorem.  We state this new definition and generalise the main results of Sections \ref{sec:lga=cpa} and \ref{sec:k-theory} to the non-normal case.

\subsection*{Acknowledgements}

The authors wish to thank the Fields Institute and the Banff International Research Station for providing the facilities
for them to discuss various aspects of this paper.

\section{Labelled spaces and their $C^*$-algebras} \label{sec:bg}

\noindent
In this section we will  briefly review the definitions of labelled spaces. We also introduce the notion of a \emph{normal} labelled space and associate a $C^*$-algebras with any weakly left-resolving normal labelled space.

\subsection*{Directed graphs}

A directed graph $E$ consists of a quadruple $( E^0 , E^1, r , s )$ where $E^0$ and $E^1$ are sets of vertices and edges, respectively and $r, s : E^1 \to E^0$ are maps giving the direction of each edge. A path $\lambda = \lambda_1 \ldots \lambda_n$ is a sequence of edges $\lambda_i \in E^1$ such that $r ( \lambda_i ) = s ( \lambda_{i+1})$ for $i=1 , \ldots , n-1$; we define $s ( \lambda ) = s ( \lambda_1 )$ and $r ( \lambda ) = r ( \lambda_n )$. The collection of paths of length $n$ in $E$ is denoted by $E^n$ and the collection of all finite paths in $E$ by $E^*$, so that $E^* = \bigcup_{n \ge 1} E^n$.

A \emph{loop} in $E$ is a path which begins and ends at the same vertex, that is $\lambda \in E^*$ with $s ( \lambda ) =  r ( \lambda )$. We say that $E$ is \emph{row-finite} if every vertex emits finitely many edges. We denote the collection of all infinite paths in $E$ by $E^\infty$.  A vertex $v \in E^0$ is an {\em sink} if $s^{-1}(v) = \emptyset$ and we define $\sink$ to be the set of all sinks in $E^0$.

\subsection*{Labelled graphs}

A \emph{labelled graph} $( E , \mathcal{L}  )$ over a countable alphabet $\mathcal{A}$ consists of a directed graph $E$ together with a labelling map $\mathcal{L}  : E^1 \to \mathcal{A}$. By replacing $\mathcal{A}$ with $\mathcal{L} (E^1)$ (if necessary) we may assume that the map $\mathcal{L} $ is onto.

Let $\mathcal{A}^*$ be the collection of all {\em words} in the symbols of $\mathcal{A}$. The map $\mathcal{L} $ extends naturally to a map $\mathcal{L}  : E^n \to \mathcal{A}^*$, where $n \ge 1$. For $\lambda = e_1 \ldots e_n \in E^n$ we set $\mathcal{L}  ( \lambda ) = \mathcal{L}  ( e_1 ) \ldots \mathcal{L} ( e_n )$. In this case the path $\lambda \in E^n$ is said to be a {\em representative} of the {\em labelled path} $\mathcal{L}  ( e_1 ) \ldots \mathcal{L}  ( e_n )$. Let $\mathcal{L} ( E^n )$ denote the collection of all labelled paths in $(E,\mathcal{L} )$ of length $n$ where we write $|\alpha|=n$ if $\alpha \in \mathcal{L}( E^n )$. The set $\mathcal{L}^* (E ) = \bigcup_{n \ge 1} \mathcal{L} ( E^n )$ is the collection of all labelled paths in the labelled graph $(E, \mathcal{L}  )$.

The labelled graph $( E , \mathcal{L}  )$ is {\em left-resolving} if for all $v \in E^0$ the map $\mathcal{L}  : r^{-1} (v) \to \mathcal{A}$ is injective. The left-resolving condition ensures that for all $v \in E^0$ the labels of the incoming edges to $v$ are all different. For $\alpha \in \mathcal{L}^* ( E )$ we put
\begin{equation*}
s_\mathcal{L}  ( \alpha ) = \{ s ( \lambda ) \in E^0 : \mathcal{L} ( \lambda ) = \alpha \} \text{ and } r_\mathcal{L}  ( \alpha ) = \{ r ( \lambda ) \in E^0 : \mathcal{L}  ( \lambda ) = \alpha \} ,
\end{equation*}
so that $r_\mathcal{L}  , s_\mathcal{L}  : \mathcal{L}^* ( E ) \to 2^{E^0}$, where $2^{E^0}$ denotes the set of subsets of $E^0$. We shall drop the subscript on $r_\mathcal{L} $ and $s_\mathcal{L} $ if the context in which it is being used is
clear.

For $A \subseteq E^0$ and $\alpha \in \mathcal{L}^* (E)$ the \emph{relative range of
$\alpha$ with respect to $A$} is defined to be
\begin{equation*}
r_{\mathcal L}(A,\alpha) = \{ r ( \lambda ) : \lambda \in E^* ,\ \mathcal{L}  ( \lambda ) = \alpha ,\ s ( \lambda ) \in A \} .
\end{equation*}

A collection $\mathcal{B}  \subseteq 2^{E^0}$ of subsets of $E^0$ is said to be \emph{closed under relative ranges for $(E , \mathcal{L} )$} if for all $A \in \mathcal{B} $ and $\alpha \in \mathcal{L}^* (E )$ we have $r (A,\alpha) \in \mathcal{B} $. If $\mathcal{B} $ is closed under relative ranges for $(E , \mathcal{L} )$, contains $r ( \alpha )$ for all $\alpha \in \mathcal{L}^* (E  )$ and is also closed under finite intersections and unions, then we say that $\mathcal{B} $ is \emph{accommodating} for $(E, \mathcal{L} )$. If $\mathcal{B}$ in addition is closed under relative complements, then we say that it is \emph{non-degenerate}.

We are particularly interested in the sets of vertices which are the ranges of words in $\mathcal{L}^* (E)$, so we form
\begin{equation*}
\mathcal{E}^- =\bigl\{ \{ v \} : v \in\sink \bigr\} \cup \bigl\{ r ( \alpha ) : \alpha \in \mathcal{L}^* (E ) \bigr\} .
\end{equation*}

\noindent
We then define $\mathcal{E}^{0,-}$ to be the smallest subset of $2^{E^0}$ which contains $\mathcal{E}^-$ and is accommodating and non-degenerate for $( E,\mathcal{L} )$. Of course, $2^{E^0}$ is the largest accommodating and non-degenerate collection of subsets for $(E,{\mathcal L})$.

\subsection*{Labelled spaces} \label{sec:labelled space}

A \emph{labelled space} consists of a triple $(E , \mathcal{L}  , \mathcal{B}  )$ where $(E , \mathcal{L}  )$ is a labelled graph and $\mathcal{B}$ is accommodating for $(E , \mathcal{L}  )$. If, in addition, $\mathcal{B}$ is non-degenerate, then we say that the labelled space $(E , \mathcal{L}  , \mathcal{B}  )$ is \emph{normal}.

A labelled space $(E , \mathcal{L}  , \mathcal{B}  )$ is \emph{weakly left-resolving} if for every $A , B \in \mathcal{B} $ and every $\alpha \in \mathcal{L}^* ( E )$ we have $r ( A , \alpha ) \cap r ( B , \alpha ) = r ( A \cap B , \alpha)$. If $( E , \mathcal{L} )$ is left-resolving then $(E, \mathcal{L} , \mathcal{B} )$ is weakly left-resolving for any accommodating $\mathcal{B} \subseteq 2^{E^0}$.

Let $(E , \mathcal{L} )$ be a labelled graph. For $A \subseteq E^0$ let $\mathcal{L} ( A E^1 ) = \{ \mathcal{L} (e) : e\in E^1,\ s(e) \in A \}$ (the set $\mathcal{L} ( A E^1 )$ was denoted $L^1_A$ in \cite{MR2304922}).

%\begin{definition}
%Let $( E , \mathcal{L}  , \mathcal{B}  )$ be a labelled space. If $\labe(A E^1)$ is finite for all $A \in {\mathcal B}$, then we say that $( E , \mathcal{L}  , \mathcal{B}  )$ is \emph{set-finite}. If for all $A \in \mathcal{B} $ and all $\ell \ge 1$ the set $\{ \mathcal{L} (\lambda) : \lambda \in E^\ell ,\ r ( \lambda ) \in A \}$ is finite, then we say that $( E , \mathcal{L} , \mathcal{B} )$ is \emph{receiver set-finite}.
%\end{definition}

\subsection*{$C^*$-algebras of normal labelled spaces} \label{CoLS}

We will now define representations of weakly left-resolving normal labelled spaces, and the $C^*$-algebra associated with a weakly left-resolving normal labelled space $(E,\mathcal{L},\mathcal{B})$. Notice that unlike \cite{MR2304922}, we require $(E,\mathcal{L},\mathcal{B})$ to be normal. The reason is that if we do not assume that $(E,\mathcal{L},\mathcal{B})$ is normal, then the $C^*$-algebra $C^*(E,\mathcal{L},\mathcal{B})$ associated with  $(E,\mathcal{L},\mathcal{B})$ might not satisfy the Gauge Invariant Uniquiness Theorem as stated in \cite[Theorem 5.3]{MR2304922} (see \cite[Remark 2.5]{BPW}). In the appendix we will give an example of what can go wrong if $(E,\mathcal{L},\mathcal{B})$ is not normal and explain how the definition of a representation of $(E,\mathcal{L},\mathcal{B})$ must be changed in order to obtain a gauge invariant uniqueness theorem for $C^*(E,\mathcal{L},\mathcal{B})$ when $(E,\mathcal{L},\mathcal{B})$ is not normal.

\begin{definition} \label{lgdef}
Let $( E , \mathcal{L}  , \mathcal{B} )$ be a weakly left-resolving normal labelled space. A \emph{representation} of $( E , \mathcal{L} , \mathcal{B} )$ in a $C^*$-algebra consists of projections $\{ p_A : A \in \mathcal{B} \}$ and partial isometries $\{ s_a : a \in \mathcal{A} \}$ with the properties that:
\begin{itemize}

\item[(i)] If $A, B \in \mathcal{B} $, then $p_A p_B = p_{A \cap B}$ and $p_{A \cup B} = p_A + p_B - p_{A \cap B}$, where $p_\emptyset = 0$.

\item[(ii)] If $a \in \mathcal{A}$ and $A \in \mathcal{B} $, then $p_A s_a = s_a p_{r ( A, a )}$.

\item[(iii)] If $a , b \in \mathcal{A}$, then $s_a^* s_a = p_{r( a )}$ and $s_{a}^*s_{b} = 0$ unless $a = b$.

\item[(iv)] For $A \in \mathcal{B} $  with $\mathcal{L}(A E^1)$ finite and $A \cap B = \emptyset$ for all $B\in\mathcal{B}$ satisfying $B\subseteq\sink$, we have
\begin{equation*}
p_A = \sum_{a \in \mathcal{L} ( A E^1 )} s_{a} p_{r( A , a )} s_{a}^* .
\end{equation*}
\end{itemize}
\end{definition}

\begin{remark}
Notice that if the directed graph $E$ contains sinks, then condition (iv) is different from condition (iv) of the original definition \cite[Definition 4.1]{MR2304922}. The original definition \cite[Definition 4.1]{MR2304922} was in error since it would lead to degeneracy of the vertex projections for sinks.
\end{remark}

\begin{remark}
Notice also that if $A$ is any set in $\bool$ with $\mathcal{L}(A E^1)$ finite and $A\cap\sink\in\mathcal{B}$, then combining (i) and (iv) we obtain the relation
\begin{equation*}
p_A = p_{A\cap\sink} + \sum_{a \in \mathcal{L} ( A E^1 )} s_{a} p_{r( A , a )} s_{a}^*,
\end{equation*}
cf. \cite[Remark 3.2]{bp3}.
\end{remark}

\begin{remark}
Suppose that $\{v\}\in\mathcal{B}$ for all $v\in\sink$. Then  for any $A\in\mathcal{B}$,  the condition $$A \cap B = \emptyset \text{ for all } B\in\mathcal{B} \text{ satisfying } B\subseteq\sink$$ that appears in (iv) is equivalent to the condition $$A\cap\sink=\emptyset.$$
\end{remark}

\begin{definition} \label{def:celb}
Let $( E , \mathcal{L}  , \mathcal{B}  )$ be a  weakly left-resolving normal labelled space. Then $C^* ( E , \mathcal{L}  , \mathcal{B}  )$ is the $C^*$-algebra generated by a universal representation of $( E , \mathcal{L} , \mathcal{B}  )$.
\end{definition}

The existence of $C^* ( E , \mathcal{L}  , \mathcal{B}  )$ is shown by an argument similar to the one given in \cite[Theorem 4.5]{MR2304922}. The existence of $C^* ( E , \mathcal{L}  , \mathcal{B}  )$ will also follow from Theorem \ref{theorem:main}. The universal property of $C^*(E, \mathcal{L}, \mathcal{B})$ allows us to define a strongly continuous action $\gamma$ of ${\bf T}$ on $C^*(E,\mathcal{L},\mathcal{B})$ called the {\em gauge action} (see \cite[Section 5]{MR2304922}).

\begin{remark}
Let $( E , \mathcal{L}  , \mathcal{B}  )$ be a  weakly left-resolving normal labelled space, and let $\{ p_A,\ s_a : A \in \mathcal{B},\ a \in \mathcal{A} \}$ be the universal representation of $( E , \mathcal{L}  , \mathcal{B}  )$ that generates $C^* ( E , \mathcal{L}  , \mathcal{B}  )$. For $\alpha=\alpha_1\dots\alpha_n\in\mathcal{L}(E^n)$ (where $n\ge 2$), we let $s_\alpha=s_{\alpha_1}\cdots s_{\alpha_n}$. Let $\alwords={\mathcal L}^*(E)\cup\{\emptyword\}$ where $\emptyword$ is a symbol not belonging to ${\mathcal L}^*(E)$ ($\emptyword$ denotes the empty word), and let $s_\emptyword$ denote the unit of the multiplier algebra of $C^*(E,{\mathcal L},{\mathcal B})$. It then follows from \cite[Lemma 4.4]{MR2304922} that we have
\begin{equation*}
C^*(E,{\mathcal L},{\mathcal B}) = \spc\{s_\alpha p_A s_\beta^* \;:\; \alpha,\beta \in \alwords,\ A \in {\mathcal B}\}.
\end{equation*}
\end{remark}

\section{Viewing Labelled Graph $C^*$-algebras as $C^*$-algebras of $C^*$-correspondences} \label{sec:lga=cpa}

\noindent
Let $(E,{\mathcal L},{\mathcal B})$ be a weakly left-resolving normal labelled space. In this section we will show how to construct a $C^*$-correspondence $X(E,{\mathcal L},{\mathcal B})$ whose Cuntz-Pimsner algebra (see \cite{MR2102572} and \cite{p}) is isomorphic to $C^*(E,{\mathcal L},{\mathcal B})$.

It should be noticed that this construction also works if $(E,{\mathcal L},{\mathcal B})$ is not normal (see the appendix).

\begin{remark}
In \cite{rs} it is shown that for every weakly left-resolving labelled space $(E',{\mathcal L}',{\mathcal B}')$ satisfying $\{\{v\}:v\in (E')^0_{\text{sink}}\}\subseteq \mathcal{B}'$, there exists a $C^*$-correspondence for which the corresponding Cuntz-Pimsner algebra is isomorphic to $C^*(E',{\mathcal L}',{\mathcal B}')$. Although the construction of the $C^*$-correspondence given in \cite{rs} is similar to the one given below, the coefficient algebra of the $C^*$-correspondence constructed in \cite{rs} does not seem to be useful for obtaining a formula for the $K$-theory of the C*-algebra of the labelled space as is done in Section \ref{sec:k-theory}.
\end{remark}

For each $A\in 2^{E^0}$ we let $\cha{A}$ denote the function defined on $E^0$ by
\begin{equation*}
\cha{A}(v)=
\begin{cases}
1&\text{if }v\in A,\\
0&\text{if }v\notin A.
\end{cases}
\end{equation*}
We will regard $\cha{A}$ as an element of the $C^*$-algebra of bounded functions on $E^0$ (when we calculate the $K$-theory we will regard $\cha{A}$ as an element of the group of functions from $E^0$ to $\mathbb{Z}$). Let $\Aalg$ denote the $C^*$-subalgebra of the $C^*$-algebra of bounded functions on $E^0$ generated by $\{\cha{A} : A \in \bool\}$.

\begin{lemma} \label{lemma:map}
Let $\{p_A : A \in \bool \}$ be a family of projections in a $C^*$-algebra $\Xalg$ such that $p_{A\cap B} = p_A p_B$, $p_{A\cup B } = p_A + p_B - p_{A \cap B}$ for all $A , B \in \bool$ and $p_\emptyset = 0$. Then there is a unique $*$-homomorphism $\phi : \Aalg \to \Xalg$ such that $\phi(\cha{A}) = p_A$ for all $A \in \bool$.
\end{lemma}

\begin{proof}
Since $\Aalg$ is generated by $\{ \cha{A} : \, A \in \bool\}$, there can be at most one $*$-homomorphism $\phi : \Aalg \to \Xalg$ satisfying $\phi(\cha{A}) = p_A$ for all $A \in \bool$.  We show that such a $*$-homomorphism exists.

Let $\mathfrak{F}$ denote the collection of finite subsets of $\bool$ which are closed under relative complements, intersections and unions. Then $\bigcup_{\bool' \in \mathfrak{F}}\spa \{ \cha{A} : A \in \bool' \}$ is dense in $\Aalg$ and so it is enough to prove that for every $\bool' \in \mathfrak{F}$ there is a $*$-homomorphism $\phi_{\bool'}: \spa \{\cha{A} : A \in \bool' \} \to \Xalg$ satisfying $\phi_{\bool'}(\cha{A}) = p_A$ for every $A \in \bool'$.

If $\bool' \in \mathfrak{F}$, then, since $\bool'$ is closed under relative complements, intersections and unions, there is a collection $F$ of mutually disjoint elements of $\bool'$ such that $\spa\{\cha{A}: A \in \bool' \} = \spa \{ \cha{A} \,: \, A \in F \}$. It follows that the map $\phi_{\bool'} :  \spa \{ \cha{A} : A \in \bool' \} \to \Xalg$ defined by
\begin{equation*}
\phi_{\bool'} \left(\sum_{A \in F}c_{A} \cha{A} \right) = \sum_{A \in F} c_{A} p_A
\end{equation*}
is a well-defined $*$-homomorphism which maps $\cha{A}$ to $p_A$ for every $A \in \bool'$.  Our result follows.
\end{proof}

\begin{lemma} \label{lemma:ideal}
Let $\ideal$ be a closed two-sided ideal of $\Aalg$. Then we have
\begin{equation*}
\ideal = \spc \{ \cha{A}: A \in \hbool,\ \cha{A} \in \ideal\}.
\end{equation*}
\end{lemma}

\begin{proof}
We certainly have $\spc\{\cha{A}: A\in\hbool,\ \cha{A}\in\ideal\}\subseteq \ideal$.  To prove the reverse inclusion, suppose that $f \in \ideal$ and let $\epsilon > 0$. Then there is a finite collection $F$ of mutually disjoint elements of $\hbool$ and a family $(c_A)_{A \in F}$ of complex numbers such that
\begin{equation*}
\left\| f - \sum_{A \in F} c_A \chi_A \right\| < \frac{\varepsilon}{2}.
\end{equation*}

Let $q_\ideal : \Aalg \to \Aalg/\ideal$ be the quotient map.  Then
\begin{equation*}
\left\| \sum_{A\in F,\cha{A}\notin\ideal}c_Aq_\ideal\left(\cha{A}\right)\right\|=
\left\|q_\ideal\left(f-\sum_{A\in F}c_A\cha{A}\right)\right\| < \frac{\varepsilon}{2},
\end{equation*}
and since  $\left(q_\ideal\left(\cha{A}\right)\right)_{A\in F,\cha{A}\notin\ideal}$ is a family of mutually orthogonal non-zero projections we must have $\left| c_A \right| < \frac{\varepsilon}{2}$ for all $A \in F$ such that $\cha{A} \notin \ideal$. It then follows that
\begin{equation*}
\left\| f- \sum_{A \in F, \cha{A} \in \ideal} c_A \cha{A} \right \| < \varepsilon
\end{equation*}
from which we may deduce that $f \in \spc \{ \cha{A} : A \in \hbool,\ \cha{A} \in \ideal \}$. Our result follows.
\end{proof}

For each $a \in \mathcal{A}$, let $X_a$ be the ideal of $\Aalg$ generated by $\cha{r(a)}$ so that $f \in X_a$ if and only if $f(v) = 0$ for all $v \in E^0 \setminus r(a)$. Since $X_a$ is an ideal, it is straightforward to see that $X_a$ is a right Hilbert $\Aalg$-module with inner product defined by $\langle f,g \rangle = f^* g$ and right action given by the usual multiplication in $\Aalg$.

We let $X(E,{\mathcal L},{\mathcal B})$ be the right Hilbert $\Aalg$-module $\oplus_{a \in \mathcal{A}} X_a$. To turn $X(E,{\mathcal L},{\mathcal B})$ into a $C^*$-correspondence we need to specify a left action of $\Aalg$ on $X(E,{\mathcal L},{\mathcal B})$, that is a $*$-homomorphism $\phi:\Aalg\to \mathscr{L}( X(E,{\mathcal L},{\mathcal B}))$ where $\mathscr{L}( X(E,{\mathcal L},{\mathcal B}))$ denotes the $C^*$-algebra of adjointable operators on $X(E,{\mathcal L},{\mathcal B})$ (see, for example, \cite{MR2102572}).

\begin{lemma} \label{lemma:phia}
For each $a \in \mathcal{A}$ there is a unique $*$-homomorphism $\phi_a : \Aalg \to X_a$ satisfying $\phi_a(\cha{A}) = \cha{r(A,a)}$ for every $A \in \bool$.
\end{lemma}

\begin{proof}
This follows from Lemma \ref{lemma:map} with $\Xalg = X_a$ and $p_A = \cha{r(A,a)}$ for $A \in \bool$.
\end{proof}

\begin{remark}
If the labelled graph $(E,{\mathcal L})$ is left-resolving, then we have
\begin{equation*}
\phi_a (f) (v) = \begin{cases} f(w) & \mbox{ if } s(r^{-1}(v) \cap {\mathcal L}^{-1}(a)) = \{w\} \\
0 &  \mbox{ if } s(r^{-1}(v) \cap {\mathcal L}^{-1}(a)) = \emptyset
\end{cases}
\end{equation*}
for all $a \in \mathcal{A}$, $v \in E^0$ and $f \in \Aalg$.
\end{remark}

\begin{lemma} \label{lem:reallyuseful}
Let $X(E,{\mathcal L},{\mathcal B})$ be the right Hilbert $\Aalg$-module $\oplus_{a \in \mathcal{A}} X_a$.
\begin{itemize}
\item[(i)] For each $f \in \Aalg$ the map
\begin{equation} \label{eq:phidef}
\varphi (f) : (x_a)_{a \in \mathcal{A}} \mapsto \left( \phi_a(f) x_a \right)_{a \in \mathcal{A}}
\end{equation}
is an adjointable operator on $X(E,{\mathcal L},{\mathcal B})$. The formula \eqref{eq:phidef} defines a $*$-homomorphism $\varphi : \Aalg \to {\mathscr L}( X(E,{\mathcal L},{\mathcal B}))$. Hence $(X(E,{\mathcal L},{\mathcal B}),\varphi)$ is a $C^*$-correspondence over $\Aalg$.

\item[(ii)] For $a \in \mathcal{A}$ let $e_a = (\delta_{a,b} \cha{r(a)})_{b \in \mathcal{A}} \in X(E,{\mathcal L},{\mathcal B})$, where $\delta_{a,b}$ is the Kronecker delta function. Then we have
\begin{equation*}
X(E,{\mathcal L},{\mathcal B})= \overline{\spa}_{\Aalg}\{e_a : a \in \mathcal{A} \}.
\end{equation*}
\end{itemize}
\end{lemma}

\begin{proof}
Since the effect of $\varphi (f)$ is to multiply each term in $X(E,{\mathcal L},{\mathcal B})$ by $\phi_a (f)$, one checks that \eqref{eq:phidef} defines an adjointable operator with $\varphi(f)^*=\varphi(f^*)$. It follows that $f \mapsto \varphi(f)$ defines a $*$-homomorphism $\varphi$ from $\Aalg$ to ${\mathscr L}(X(E,{\mathcal L},{\mathcal B}))$.

The second statement holds since, for each $x = (x_a)_{a \in \mathcal{A}}\in X(E,{\mathcal L},{\mathcal B})$ and each $\varepsilon>0$, we have $||x-\sum_{a\in F}e_ax_a||<\varepsilon$ for some finite subset $F$ of $\mathcal{A}$.
\end{proof}

Let ${\mathscr K}(X(E,{\mathcal L},{\mathcal B}))$ denote the ideal of ${\mathscr L}(X(E,{\mathcal L},{\mathcal B}))$ consisting of generalized compact operators (see, for example, \cite{MR2102572}). In order to define the Cuntz-Pimsner algebra ${\mathcal O}_{(X(E,\labe,\bool),\varphi)}$ associated with $(X(E,\labe,\bool),\varphi)$, we must first characterise the ideal
\begin{equation*}
	J_{(X(E,{\mathcal L},{\mathcal B}),\varphi)}=\varphi \inv ( \Kalg(X(E,\labe,\bool))) \cap \ker(\varphi)^\perp
\end{equation*}
of $\Aalg$ (cf.\ \cite[Definition 3.2]{MR2102572}). To do this, we let
\begin{equation} \label{eq:boolj}
\boolj := \{ A \in \hbool : \mathcal{L}(A E^1)\text{ is finite and }A\cap B = \emptyset\text{ for all }B\in\mathcal{B}\text{ satisfying }B\subseteq\sink\}.
\end{equation}

\begin{lemma}\label{lemma:j}
The ideal $\Jalg$ of $\Aalg$ is given by
\begin{equation*}
\Jalg = \spc \{ \cha{A} : A \in \hbool_J  \}.
\end{equation*}
\end{lemma}

\begin{proof}
By Lemma \ref{lemma:ideal} it is enough to prove that for $A \in \hbool$, we have $\chi_A \in \Jalg$ if and only if $A \in \hbool_J$.  Since $\ker(\varphi)=\bigcap_{a\in\mathcal{A}}\ker(\phi_a)$, it follows from Lemma \ref{lemma:ideal} that
\begin{equation*}
\ker(\varphi)= \spc\{\cha{B}:B\in\bool,\ r(B,a)=\emptyset\text{ for all }a\in\mathcal{A}\}.
\end{equation*}
Since $r(B,a)=\emptyset$ for all $a\in\mathcal{A}$, if and only if $B\subseteq\sink$, it follows that
\begin{equation*}
\ker(\varphi)^\perp= \spc\{ A \in \hbool : A\cap B = \emptyset\text{ for all }B\in\mathcal{B}\text{ satisfying }B\subseteq\sink\}.
\end{equation*}

We claim that for each $\eta \in \Kalg(X( E,\labe,\bool))$ the set $\{a \in \mathcal{A} : \| \langle e_a, \eta( e_a ) \rangle \| \ge 1 \}$ is finite. It suffices to check our claim for $\eta = \theta_{x,y}$ since these elements form a spanning set for $\Kalg(X(E,\labe,\bool))$.  For each $a_0 \in \mathcal{A}$ and $x = (x_a),y = (y_a) \in X(E,{\mathcal L},\bool)$ one checks that $\langle e_{a_0} , \theta_{x,y} e_{a_0} \rangle = x_{a_0}^* y_{a_0}$, and since $\|x_a^*y_a\|\ge 1$ for only a finite number of $a$'s, our claim follows. So, if $A \in \hbool$ and $\mathcal{L}(A E^1)$ is infinite, then $\varphi( \cha{A} ) \notin  \Kalg( X(E,\labe,\bool) )$.

Conversely, if $ A\in \hbool$ and $\mathcal{L}(A E^1)$ is finite, then we have
\begin{equation}\label{eqn:compactphi}
\varphi(\cha{A})=\sum_{a\in \mathcal{L}(A E^1)}\theta_{e_a,e_a\cha{r(A,a)}} \in\Kalg(X(E,\labe,\bool)).
\end{equation}
Thus, if $A \in \hbool$, then $\varphi(\cha{A})\in  \Kalg(X(E,\labe,\bool)) $ if and only if $\mathcal{L}(A E^1)$ is finite. The result follows by definition of $\Jalg$ and $\boolj$.
\end{proof}

Recall from \cite[Definition 3.4]{MR2102572} that a representation of the $C^*$-correspondence $(X(E,\labe,\bool),\varphi)$ on a $C^*$-algebra $\Xalg$ consists of a pair $(\pi,t)$ where $\pi:  \Aalg \to \Xalg$ is a $*$-homomorphism and $t: X(E,\labe,\bool) \to \Xalg$ is a linear map satisfying
\begin{enumerate}
\item $t(x)^*t(y) = \pi \left(\langle x,y \rangle \right)$ for $x,y \in X(E,\labe,\bool)$ and
\item $\pi(f) t(x) = t(\varphi(f) x)$ for $f \in \Aalg$ and $x \in X(E,\labe,\bool)$.
\end{enumerate}
Following \cite[Definition 2.3]{MR2102572} we define a $*$-homomorphism $\psi_t:  \Kalg(X(E,\labe,\bool)) \to \Xalg$ by
\begin{equation*}
\psi_t(\theta_{x,y}) = t(x) t(y)^* \mbox{ for } x, y \in X(E,\labe,\bool).
\end{equation*}
Moreover, such a representation is called {\em covariant} if, in addition, we have $\pi(f) = \psi_t ( \varphi(f) )$ for all $f \in \Jalg$.  The \emph{Cuntz-Pimsner algebra} ${\mathcal O}_{(X(E,\labe,\bool),\varphi)}$ of $(X(E,\labe,\bool),\varphi) $ is defined in \cite{MR2102572} (see also \cite{p}) to be the $C^*$-algebra generated by a universal covariant representations of $(X ( E , \labe , \bool ), \varphi )$.

\begin{theorem} \label{theorem:main}
Let $( E, \labe, \bool )$ be a weakly left-resolving normal labelled space. Then there is a one-to-one correspondence between covariant representations of $(X ( E , \labe , \bool ), \varphi )$ and representations of $( E, \labe , \bool )$ that takes a covariant representation $(\pi,t)$ of $(X ( E , \labe , \bool ), \varphi )$ to the representation $\{\pi(\cha{A}),\ t(e_a) : A\in \bool,\ a \in \mathcal{A} \}$ of $( E, \labe , \bool )$, and so $\Oalg_{( X ( E,\labe, \bool),\varphi)} \cong C^*( E , \labe , \bool )$.
\end{theorem}

\begin{proof}
Let $(\pi,t)$ be a covariant representation of $(X ( E, \labe, \bool),\varphi )$. We claim that \begin{equation*}\{ \pi(  \cha{A}) ,\ t( e_a ) : A \in \bool,\ a \in \mathcal{A} \}\end{equation*} is a representation of  $(E , \labe, \bool )$.

It is straightforward to check that $\{\pi( \cha{A}) ,\ t(e_a) : A \in \bool,\ a \in \mathcal{A} \}$ satisfies (i)--(iii) of Definition \ref{lgdef}. We will now show that  $\{\pi( \cha{A}) ,\ t(e_a) : A \in \bool,\ a \in \mathcal{A} \}$ also satisfies (iv) of Definition \ref{lgdef}. Assume $A\in\bool$, $\mathcal{L}(AE^1)$ is finite and that $A\cap B=\emptyset$ for any $B\in\bool$ satisfying $B\subseteq\sink$. It follows from Lemma \ref{lemma:j} that $\cha{A} \in \Jalg$. By covariance and Equation \eqref{eqn:compactphi} we have
\begin{equation*}
\pi(\cha{A}) = \psi_t(\phi(\cha{A})) = \psi_t \Bigl(\sum_{a \in {\mathcal L}(AE^1)} \theta_{e_a , e_a \cha{r(A,a)}} \Bigr)= \sum_{a \in \mathcal{L}(AE^1)} t(e_a) \pi(\cha{r(A,a)}) t(e_a)^*.
\end{equation*}
Thus $\{\pi(\cha{A}),\ t(e_a) : A\in \bool,\ a \in \mathcal{A} \}$ satisfies (iv) of Definition \ref{lgdef}, and is therefore a representation of $(E,\labe,\bool)$.

On the other hand, suppose that $\{ P_A ,\ S_a : A \in \bool,\ a \in \mathcal{A} \}$ is
a representation of $(E , \labe , \bool )$ in a $C^*$-algebra $\Xalg$. By definition
$\Aalg = \spc\{\cha{A}: A\in\bool\}$ and by Lemma \ref{lem:reallyuseful},
we have
$X(E,\labe,\bool) = \spc_{\Aalg}\{e_a: a\in\labe (E^1)\}$. Hence there can
at most be one representation $(\pi,t)$ of $(X(E,\labe,\bool),\varphi)$
such that $\pi ( \cha{A} ) = P_A$ for $A \in \bool$ and $t( e_a ) = S_a$ for $a \in \mathcal{A}$.
We now construct such a representation.

By Lemma \ref{lemma:map} there is a unique
$*$-homomorphism $\pi : \Aalg \to \Xalg$ with
$\pi ( \cha{A} ) = p_A$ for $A \in \bool$. We now construct $t : (X(E,\labe,\bool),\varphi) \to \Xalg$.
Given a finite subset $F$ of $\mathcal{A}$ and a family $(f_a)_{a\in F}$ of functions with
$f_a\in X_a$ for each $a\in F$, define
\begin{equation*}
t \Bigl( \sum_{a \in F} e_a f_a \Bigr) =\sum_{a\in F} S_a \pi(f_a) \in \Xalg.
\end{equation*}
If  $(f_a)_{a\in F}$ and $(g_a)_{a\in F}$ are families of functions with $f_a, g_a \in X_a$ for each
$a \in F$, then
\begin{align*}
t \Bigl(\sum_{a \in F} e_a f_a \Bigr)^* t \Bigl( \sum_{a \in F} e_a g_a \Bigr) &= \Bigl( \sum_{a \in F} S_a \pi ( f_a ) \Bigr)^*
\Bigl( \sum_{a \in F} S_a \pi(g_a) \Bigr) = \sum_{a \in F}\pi(f_a)^* P_{r(a)}\pi(g_a) \\
\intertext{by (iii) of Definition \ref{lgdef}. Since $P_{r(a)} = \pi ( \cha{r(a)} )$ we have}
t \Bigl(\sum_{a \in F} e_a f_a \Bigr)^* t \Bigl( \sum_{a \in F} e_a g_a \Bigr) &= \sum_{a \in F} \pi ( f_a^* g_a) =
\pi\biggl( \Bigl \langle \sum_{a \in F} e_a f_a, \sum_{a\in F} e_a g_a
\Bigr \rangle \biggr).
\end{align*}
Thus $t$ extends to a linear map from $X (E,\labe,\bool)$ to $\Xalg$
which satisfies $t(x)^*t(y)=\pi ( \langle x , y \rangle )$ for $x,y \in
X(E,\labe,\bool)$. We claim that for $x \in X (E,\labe,\bool)$  and $f \in \Aalg$
we have $\pi(a) t(x) = t( \varphi (a) x)$.

Let $A\in\bool$ and $a\in\mathcal{A}$, then
\begin{equation} \label{eq:neededbelow}
\pi(\cha{A}) S_a = P_A S_a = S_a P_{r(A,a)} = S_a \pi ( \cha{r(A,a)})=
S_a \pi( \phi_a (\cha{A}) ).
\end{equation}
Since $\Aalg=\spc\{\cha{A}: A\in\bool\}$, it follows from Equation \eqref{eq:neededbelow} that
\begin{equation*}
\pi(f) t( e_a ) = \pi(f) S_a = S_a \pi ( \phi_a (f) ) = t ( e_a ) \pi ( \phi_a (f) ) = t ( \varphi(f) e_a )
\end{equation*}
for all $f \in \Aalg$ and all $a \in\mathcal{A}$. Our claim then follows from Lemma \ref{lem:reallyuseful} (ii).
Hence we have constructed a representation $(\pi,t)$ of $(X( E, \labe, \bool ), \varphi )$. It remains to
check covariance.

Let $A\in\boolj$, then by Equation \eqref{eqn:compactphi} and \eqref{eq:neededbelow} we have
\begin{equation*}
\begin{split}
\pi(\cha{A}) &= P_A = \sum_{a\in \mathcal{L}(AE^1)}S_a P_{r(A,a)} S_a^*
= \sum_{a\in {\mathcal L}( AE^1 )}t ( e_a )\pi ( \cha{r(A,a)} ) t ( e_a )^* \\
&= \sum_{a\in {\mathcal L}(AE^1)}\psi_t ( \theta_{e_a,e_a \phi_a (\cha{A}) } ) =
\psi_t ( \varphi ( \cha{A}) ).
\end{split}
\end{equation*}
By Lemma \ref{lemma:j} we have $\Jalg = \spc\{ \cha{A} :
A \in \boolj\}$ and  it follows that $\pi (f) = \psi_t( \varphi(f))$ for all
$f\in\Jalg$. Thus $(\pi,t)$ is a covariant representation
of $( X ( E, \labe,\bool), \varphi )$.

The final statement follows by the universal nature of both algebras.
\end{proof}

In the appendix, it is explained how Theorem \ref{theorem:main} can be generalised to the case where $(E,\labe,\bool)$ is not normal (see Remark \ref{remark:nn}).

From Theorem \ref{theorem:main} we may deduce a number of corollaries. The second of these is a generalization of the Gauge Invariant Uniqueness Theorem found in \cite{BPW}.

\begin{corollary}  \label{corol:iota}
Let $(E,\labe,\bool)$ be a weakly left-resolving normal labelled space. Then there exists an injective $*$-homomorphism $\iota_{\Aalg}:\Aalg\to C^*(E,\labe,\bool)$ such that $\iota_{\Aalg}(\cha{A}) = p_A$ for every $A \in \bool$.
\end{corollary}

\begin{proof}
Follows from Theorem \ref{theorem:main} and \cite[Proposition 4.11]{MR2102572}.
\end{proof}

In the appendix, it is explained how Corollary \ref{corol:iota} can be generalised to the case where $(E,\labe,\bool)$ is not normal (see Remark \ref{remark:nn}).

\begin{corollary}  \label{cor:giut}
Let $(E,\labe,\bool)$ be a weakly left-resolving normal labelled space. Let $\{ p_A,\ s_a : A \in \mathcal{B},\ a \in \mathcal{A} \}$ be the universal representation of $( E , \mathcal{L}  , \mathcal{B}  )$ that generates $C^* ( E , \mathcal{L}  , \mathcal{B}  )$, let $\{ q_A,\ t_a : A \in \mathcal{B},\ a \in \mathcal{A} \}$ be a representation of $( E , \mathcal{L}  , \mathcal{B}  )$ in a $C^*$-algebra $\Xalg$ and let $\pi$ be the unique $*$-homomorphism from $C^* ( E , \mathcal{L}  , \mathcal{B}  )$ to $\Xalg$ that maps each $p_A$ to $q_A$ and each $s_a$ to $t_a$. Then $\pi$ is injective if and only if $q_A$ is non-zero whenever $A\ne\emptyset$, and  for each $z\in\mathbb{T}$ there exists a $*$-homomorphism $\beta_z:C^*(\{ q_A,\ t_a : A \in \mathcal{B},\ a \in \mathcal{A} \})\to C^*(\{ q_A,\ t_a : A \in \mathcal{B},\ a \in \mathcal{A} \})$ such that $\beta_z(q_A)=q_A$ and $\beta_z(t_a)=zt_a$ for $A\in\bool$ and $a\in\mathcal{A}$.
\end{corollary}

\begin{proof}
Follows from Lemma \ref{lemma:ideal}, Theorem \ref{theorem:main}, and  \cite[Theorem 6.4]{MR2102572}.
\end{proof}

In the appendix, Corollary \ref{cor:giut} is generalised to the case where $(E,\labe,\bool)$ is not normal (see Corollary \ref{cor:giutm}).

\begin{corollary} \label{cor:nuclear}
Let $(E,\labe,\bool)$ be a weakly left-resolving normal labelled space. Then $C^*(E,\labe,\bool)$ is nuclear. If, in addition, $\bool$ and $\mathcal{A}$ are countable, then $C^*(E,\labe,\bool)$ satisfies the Universal Coefficient Theorem of \cite{MR894590}.
\end{corollary}

\begin{proof}
The first part follows from \cite[Corollary 7.4]{MR2102572} and the fact that $\Jalg$ and $\Aalg$ are commutative and thus nuclear. If, in addition, $\bool$ and $\mathcal{A}$ are countable, then $\Aalg$ and $(X(E,\labe,\bool),\varphi)$ are separable.  It then follows from \cite[Proposition 8.8]{MR2102572} that $C^*(E,\labe,\bool)$ satisfies the Universal Coefficient Theorem of \cite{MR894590}.
\end{proof}

In the appendix, it is explained how Corollary \ref{cor:nuclear} can be generalised to the case where $(E,\labe,\bool)$ is not normal (see Remark \ref{remark:nn}).

\section{$K$-theory} \label{sec:k-theory}

We now compute the $K$-theory of $C^*(E,\labe,\bool)$ when $(E,\labe,\bool)$ is a weakly left-resolving normal labelled space (see the appendix for a procedure for computing the $K$-theory when $(E,\labe,\bool)$ is not normal). Our approach is to use \cite[Theorem 8.6]{MR2102572} which involves detailed knowledge of the map $[X]: K_*(\Jalg)\to K_*(\Aalg)$. Since we will need to work explicitly with this map, we will now recall its definition in detail (cf. \cite[Definition 8.3]{MR2102572}).

Let $D_{X(E,\labe,\bool)}$ denote the linking algebra $\Kalg(X(E,\labe,\bool)\oplus\Aalg)$. Following \cite[Definition B.1]{MR2102572} we denote the natural embedding of $\Aalg$ into $D_{X(E,\labe,\bool)}$ by $\iota_{\Aalg}$ and the natural embedding of $\Kalg(X(E,\labe,\bool))$ into $D_{X(E,\labe,\bool)}$ by $\iota_{\Kalg(X(E,\labe,\bool))}$. It is shown in \cite[Proposition B.3]{MR2102572} that the map $(\iota_{\Aalg})_*:K_*(\Aalg)\to K_*(D_{X(E,\labe,\bool)})$ induced by $\iota_{\Aalg}$ is an isomorphism. The map $[X]: K_*(\Jalg)\to K_*(\Aalg)$ is then defined to be $(\iota_{\Aalg})_*\inv \circ(\iota_{\Kalg(X(E,\labe,\bool))})_* \circ(\varphi|_{\Jalg})_*$.

\begin{lemma} \label{lemma:keqv}
Let $(E,\labe,\bool)$ be a weakly left-resolving normal labelled space. Let $a \in \mathcal{A}$ and $A \in \hbool$ such that $A \subseteq r(a)$. Then for $[ \cha{A} ]_0 \in K_0(\Aalg)$ and $[ \theta_{e_a,e_a \cha{A}} ]_0 \in K_0(\Kalg(X(E, \labe, \bool)))$ we have
\begin{equation*}
( \iota_{\Aalg})_*([  \cha{A} ]_0 ) = ( \iota_{\Kalg( X ( E, \labe, \bool ) )} )_*([ \theta_{e_a,e_a \cha{A}} ]_0)
\end{equation*}
in $K_0 ( D_{X ( E , \labe , \bool )} )$.
\end{lemma}

\begin{proof}
Let $v$ be the function that maps $( \eta, f ) \in X( E, \labe, \bool ) \times \Aalg$ to $(e_a \cha{A} f, 0) \in X( E, \labe, \bool) \times \Aalg$. One checks that $v \in \Kalg( X( E, \labe, \bool ) \oplus \Aalg )$, that $vv^* = \iota_{ \Kalg( X( E, \labe, \bool ))}( \theta_{e_a, e_a \cha{A}})$ and $v^*v = \iota_{\Aalg}( \cha{A} )$.  It follows that the elements $\iota_{ \Kalg( X( E, \labe, \bool ))}( \theta_{e_a, e_a \cha{A}})$ and $\iota_{ \Aalg }( \cha{A} )$ are equivalent in $K_0( D_{ X( E, \labe, \bool)})$.
\end{proof}

Since $\Aalg = \spc \{ \cha{A} : A \in \hbool \}$ and $\Jalg = \spc \{ \cha{A} : A \in \boolj \}$, it follows that there is an isomorphism from $\spa_\Z \{ \cha{A} :  A \in \hbool \}$ to $K_0(\Aalg)$ which for each $A\in\hbool$ takes $\cha{A}$ to $[\cha{A}]_0$, and an isomorphism from $\spa_\Z \{ \cha{A} :  A \in \boolj \}$ to $K_0(\Jalg)$ which for each $A\in\boolj$ takes $\cha{A}$ to $[\cha{A}]_0$. We will simply identify $K_0(\Aalg)$ with $\spa_\Z \{ \cha{A} :  A \in \hbool \}$, $K_0(\Jalg)$ with $\spa_\Z \{ \cha{A} :  A \in \boolj \}$ and each $[\cha{A}]_0$ with $\cha{A}$.

\begin{lemma} \label{lemma:ktheory}
Identifying $K_0(\Aalg)$ with $\spa_\Z \{ \cha{A} :  A \in \hbool \}$ and $K_0(\Jalg)$ with $\spa_\Z \{ \cha{A} :  A \in \boolj \}$, the homomorphism $[X] : K_0(\Jalg) \to K_0(\Aalg)$ induces the map $\Phi: \spa_\Z \{ \cha{A} :  A \in \boolj \} \to \spa_\Z \{ \cha{A} :  A \in \hbool \}$ determined by
\begin{equation} \label{eq:1}
\cha{A}  \mapsto \sum_{ a \in  \mathcal{L}(AE^1)} \cha{ r(A,a)} \mbox{ for } A \in \boolj.
\end{equation}
\end{lemma}

\begin{proof}
Let $A\in\boolj$. Since $\phi_a( \cha{A} ) = \cha{r(A,a)}$ is a projection in $\Aalg$, we have $\varphi(\cha{A}) = \sum_{ a \in {\mathcal L}(AE^1)} \theta_{ e_a, e_a \cha{r(A,a)}}$ from Equation \eqref{eqn:compactphi}. It follows from Lemma \ref{lemma:keqv} that
\begin{equation*}
\begin{split}
( \iota_{\Kalg(X(E,\labe,\bool))} )_* \circ ( \varphi|_{ \Jalg })_*([ \cha{A} ]_0)
& = ( \iota_{ \Kalg(X(E,\labe,\bool))})_* \left( \left[ \sum_{a\in {\mathcal L}(AE^1)} \theta_{ e_a, e_a \cha{r(A,a)}} \right]_0 \right)\\
&= \sum_{a \in {\mathcal L}(AE^1)} (\iota_{\Kalg(X(E,\labe,\bool))})_* ([ \theta_{ e_a, e_a \cha{r(A,a)}} ]_0) \\
&= \sum_{ a \in \mathcal{L}(A E^1)} ( \iota_{\Aalg} )_*([ \cha{r(A,a)}]_0 ),
\end{split}
\end{equation*}
from which Equation \eqref{eq:1} follows using the given identifications and the definition of $[X]$.
\end{proof}

If $\mathcal{B} \subseteq 2^{E^0}$ is uncountable, then $\Aalg$ is not separable. However, it is still locally finite dimensional. Such nonseparable AF algebras were considered in
\cite{katsura}.

\begin{lemma} \label{lem:K1=0}
Let $(E,\labe,\bool)$ be a weakly left-resolving normal labelled space. Then $\Aalg$ and $\Jalg$ are locally finite dimensional algebras with $K_1 ( \Aalg ) =K_1(\Jalg )= 0$.
\end{lemma}

\begin{proof}
We prove the result for $\Aalg$ and note that the proof for $\Jalg$ is similar. The first statement follows by the argument given in the proof of \cite[Theorem 5.3]{MR2304922}. To prove that $K_1 ( \Aalg ) = 0$, we argue as follows\footnote{Thanks to Iain Raeburn for showing us how to do this.}. By adding $E^0$ to $\mathcal{B}$ (if necessary) we may assume that $\Aalg$ is unital. Let $u$ be a unitary in $\Aalg$, then given $0 < \epsilon < 1$ there is a unitary $u'$ in $\spa_\Z \{ \cha{A} :  A \in\hbool \}$ such that $\| u - u' \| < \epsilon$. Using the local finite dimensionality of $\Aalg$ one can show that $u'$ belongs to a finite dimensional subalgebra $U$ of $\Aalg$ which contains the identity $1_{\Aalg}$. Since the unitary group of a finite dimensional $C^*$-algebra is connected there is a continuous path of unitaries from $u'$ to $1_{\Aalg} = \cha{E^0}$ which remains continuous when we consider $U$ as a subalgebra of $\Aalg$. Since $\| u -  u' \| < \epsilon < 1$ there is a continuous path of unitaries from $u$ to $u'$ and hence from $u$ to $1_{\Aalg}$. The result now follows.
\end{proof}

\begin{theorem} \label{theorem:ktheory}
Let $(E,\labe,\bool)$ be a weakly left-resolving normal labelled space. Let $(1-\Phi)$ be the linear map from $\spa_\Z\{\cha{A}: A\in\boolj\}$ to $\spa_\Z \{ \cha{A} :  A \in \hbool \}$ given by
\begin{equation*}
(1-\Phi) ( \cha{A} ) = \cha{A} - \sum_{a \in \labe(AE^1)} \cha{r(A,a)}, \quad  A \in \boolj.
\end{equation*}
Then $K_1(C^*(E,\labe,\bool)) \cong \ker(1-\Phi)$, and $K_0(C^*(E,\labe,\bool)) \cong \spa_\Z\{\cha{A}: A\in\hbool\}/\im(1-\Phi)$ via $[p_A]_0 \mapsto \cha{A}+\im(1-\Phi)$ for $A \in \hbool$.
\end{theorem}

\begin{proof}
The result follows by \cite[Theorem 8.6]{MR2102572}, Theorem \ref{theorem:main}, Lemma \ref{lemma:ktheory} and the fact that $K_1(\Aalg)=K_1(\Jalg )=0$ from Lemma \ref{lem:K1=0}.
\end{proof}

In the appendix, Theorem \ref{theorem:ktheory} is generalises to the case where $(E,\labe,\bool)$ is not normal (see Theorem \ref{thm:ktheorym}).

\section{Examples} \label{sec:examples}

\noindent
In this section we show that Theorem \ref{theorem:ktheory} unifies the $K$-theory formulas for several interesting classes of $C^*$-algebras.

\subsection*{Example 1 -- Graph algebras}

Let $E$ be a directed graph. Following \cite[Examples 3.3 and 4.3 (i)]{MR2304922} we may realise $C^*(E)$ as the $C^*$-algebra of a left-resolving normal labelled space in the following manner. We take the trivial labelling ${\mathcal L}_t:  E^1 \to E^1$ defined for $e \in E^1$ by ${\mathcal L}_t(e) = e$.  Then $(E,\mathcal{L})$ is a left-resolving labelled graph and ${\mathcal E}^{0,-}$ is the set of all finite subsets of $E^0$. Since $(E,\mathcal{L})$ is left-resolving, it follows that the labelled space $(E,{\mathcal L}_t,{\mathcal E}^{0,-})$, in addition to being normal, is weakly left-resolving. By \cite[Proposition 5.1 (i)]{MR2304922} there is an isomorphism $\phi: C^*(E) \to C^*(E,{\mathcal L}_t,{\mathcal E}^{0,-})$ that maps $P_v$ to $p_{\{v\}}$ and $S_e$ to $s_e$ where $\{ S_e ,\ P_v : e \in E^1,\ v \in E^0 \}$ and $\{ s_e ,\ p_A : e \in E^1,\ A \in \mathcal{E}^{0,-} \}$ are the canonical generators of $C^*(E)$ and $C^*( E, \mathcal{L}_t , \mathcal{E}^{0,-})$, respectively.

Let $E^0_{ns} = \{v \in E^0 : 0 < | s^{-1}(v) | < \infty \}$ denote the set of nonsingular vertices in $E^0$.  For each $v \in  E^0_{ns}$ we let $e_v \in \bigoplus_{v \in E^0_{ns}} \Z$  be defined by $e_v = ( \delta_{v,w} )_{w \in E^0_{ns}}$.  We similarly define $e_v \in \bigoplus_{v \in E^0} \Z$ for each $v \in E^0$.

Note that $\mathcal{E}^{0,-}_J$ is precisely the collection of all finite subsets of $E^0_{ns}$.  One checks that there are isomorphisms $\psi :  \spa_\Z \{ \cha{A} : A \in \hbool \} \to \bigoplus_{ v \in E^0 } \Z$ and $\psi_J: \spa_\Z \{ \cha{A} : A \in \mathcal{E}^{0,-}_J \} \to \bigoplus_{ v \in E^0_{ns}} \Z$ such that for all $v \in E^0$ (respectively\ $v \in E^0_{ns}$) we have $\psi( \cha{v} ) = e_v$ (respectively\ $\psi_J( \cha{v} ) = e_v$).  For $v \in E^0_{ns}$ we have  $\Phi( \cha{v} )=\sum_{ f \in  s \inv \{ v \} } \cha{ r(f) }$ and we may now deduce the following Corollary from Theorem \ref{theorem:ktheory}.

\begin{corollary}[{cf.\ \cite[Theorem 3.1]{MR1936076}}]
Let $E$ be a directed graph. With notation as above, define a linear map $(1 - \Phi) : \bigoplus_{ v \in E^0_{ns} } \Z \to \bigoplus_{ v \in E^0 } \Z$ by
\begin{equation*}
(1- \Phi) (e_v)  = e_v - \sum_{f \in s \inv \{v\} } e_{ r(f) },\quad v\in E^0_{ns}.
\end{equation*}
Then $K_1(C^*(E))$ is isomorphic to $\ker(1-\Phi)$ and there exists an isomorphism from $K_0(C^*(E))$ to $\coker(1 - \Phi)$ which maps $[ P_v ]_0$ to $e_v + \im( 1 - \Phi )$ for each $v \in E^0$.
\end{corollary}

\subsection*{Example 2 -- Ultragraph algebras}

For the definition and properties of an ultragraph see \cite{tom}. Following \cite[Examples 3.3 and 4.3 (ii)]{MR2304922} we may realise an ultragraph as a labelled graph as follows: Let ${\mathcal G} = (G^0,{\mathcal G}^1,r,s)$ be an ultragraph. We define a directed graph $E_{\mathcal G} = (E_{\mathcal G}^0,E_{\mathcal G}^1,r',s')$ by putting $E_{\mathcal G}^0 = G^0$, $E_{\mathcal G}^1 = \{ (e,w) : e \in {\mathcal G}^1,\ w  \in r(e) \}$, $r'(e,w) = w$ and $s'(e,w) = s(e)$. Setting ${\mathcal A} = {\mathcal G}^1$ we may define a labelling map ${\mathcal L}_{\mathcal G} : E^1 \to {\mathcal A}$ by ${\mathcal L}_{\mathcal G} (e,w) = e$.  If we let ${\mathcal G}^0$ be as defined in \cite[Definition 2.4]{arXiv:math/0611318}, then $(E_{\mathcal G}, {\mathcal L}_{\mathcal G},{\mathcal G}^0)$ is a normal weakly left-resolving labelled space. As in \cite[Proposition 5.1 (ii)]{MR2304922}, using the gauge invariant uniqueness theorem \cite[Theorem 6.8]{tom} there is an isomorphism $\phi : C^*({\mathcal G}) \to C^*(E_{\mathcal G}, {\mathcal L}_{\mathcal G}, {\mathcal G}^0)$ that maps $P_A$ to $p_A$ and $S_e$ to $s_e$ where $\{ S_e , P_A \}$ and $\{ s_e , p_A \}$ are the canonical generators of $C^*({\mathcal G})$ and $C^*(E_{\mathcal G},{\mathcal L}_{\mathcal G},{\mathcal G}^0)$, respectively.

The following definitions are taken from \cite{arXiv:math/0611318}. Let $\mathfrak{A}_\mathcal{G}$ denote the $C^*$-subalgebra of $\ell^\infty (G^0)$ generated by the point masses $\{ \delta_v :  v \in G^0 \}$ and the characteristic functions $\{ \cha{ r(e) }  :  e \in \mathcal{G}^1 \}$ and let $Z_\mathcal{G}$ denote the algebraic subalgebra of $\ell^\infty ( G^0, \Z )$ generated by these functions.  Let $G^0_{rg} \subseteq G^0$ denote the set of vertices $v \in G^0$ satisfying $0 < | s^{-1}(v) | < \infty$.   For $v \in G^0_{rg}$ let $e_v \in \bigoplus_{v \in G^0_{rg}}\Z$ be defined by $e_v = ( \delta_{v,w} )_{w \in G^0_{rg}}$ and define $e_v \in Z_\mathcal{G} \subseteq \ell^\infty (G_0 , \Z)$ similarly for $v \in E^0$.

It follows from \cite[Proposition 2.6 and 2.20]{arXiv:math/0611318} that $Z_\mathcal{G} = K_0( \mathfrak{A}_\mathcal{G} )=\spa_\Z \{ \cha{A} : A \in {\mathcal G}^0 \}$. One checks that $A \in {\mathcal G}^0$ belongs to ${\mathcal G}^0_J$ if and only if $A$ is a finite subset of $G^0_{rg}$.  It then follows that there is an isomorphism from $\bigoplus_{ v \in G^0_{rg} } \Z$ to $\spa_\Z \{ \cha{A} \,: \, A \in {\mathcal G}^0_J \}$ which maps $e_v$ to $\cha{ \{ v \} }$ for every $v \in G^0_{rg}$. If $v \in G^0_{rg}$, then $\Phi( \cha {\{ v \} })=\sum_{ f \in  s\inv \{v\} } \cha{r(f)}$, and we may deduce the following Corollary to Theorem \ref{theorem:ktheory}.

\begin{corollary}[{cf. \cite[Theorem 5.4]{arXiv:math/0611318}}]
Let $\mathcal{G}=(G^0,\mathcal{G}^1,r,s)$ be an ultragraph. With notation as above, define a linear map $(1 - \Phi ) : \bigoplus_{ v \in G^0_{rg} } \Z \to Z_\mathcal{G}$  by
\begin{equation*}
( 1 - \Phi ) (e_v)= e_v - \sum_{ f \in s\inv \{ v \} } \cha{ r(f) }, \quad v \in G^0_{rg}.
\end{equation*}
Then $K_1(C^*(\mathcal{G}))$ is isomorphic to $\ker(  1 - \Phi )$ and there is an isomorphism from $K_0( C^*( \mathcal{G} ))$ to $\coker( 1 - \Phi )$ which maps $[P_A]_0$ to $\cha{A} + \im(1 - \Phi)$ for each $ A \in \mathcal{G}^0$ .
\end{corollary}

\subsection*{Example 3 -- Matsumoto algebras}

We are interested in two $C^*$-algebras which Matsumoto associates to a (two-sided) shift space $\Lambda$ over a finite alphabet $\mathcal{A}$. The first, which we denote $\mathcal{O}_{\Lambda^*}$, was described in \cite{MR1454478}, \cite{MR1646513}, \cite{MR1691469}, \cite{MR1688137}, \cite{MR1774695} and \cite{MR1764930}. It follows from an argument similar to the one used in \cite[Theorem 6.3]{MR2304922} (but using Corollary \ref{cor:giut} instead of \cite[Theorem 5.3]{MR2304922}) that $\mathcal{O}_{\Lambda^*} \cong C^* ( E_{\Lambda^*} , \mathcal{L}_{\Lambda^*} , \mathcal{E}_{\Lambda^*}^{0,-} )$ where $( E_{\Lambda^*} , \mathcal{L}_{\Lambda^*} )$ is the predecessor graph (or past set cover) of $\Lambda$, as defined in \cite[Examples 3.3]{MR2304922} (vii) (see also \cite[Definition 5.2]{bep}).

The second, which we denote $\mathcal{O}_\Lambda$ was described in \cite{MR2091486} and \cite{MR1852456}. Under a technical hypothesis, analogous to Cuntz and Krieger's condition (I), it follows from an argument similar to the one used in \cite[Theorem 6.3]{MR2304922} (again using Corollary \ref{cor:giut} instead of \cite[Theorem 5.3]{MR2304922}) that $\mathcal{O}_{\Lambda} \cong C^* ( E_{\Lambda} , \mathcal{L}_{\Lambda} , \mathcal{E}_{\Lambda}^{0,-} )$ where $( E_{\Lambda} , \mathcal{L}_{\Lambda} )$ is the left Krieger cover of $\Lambda$, as defined in \cite[Section 2]{MR0776312} (see also \cite[Examples 3.3 (vi)]{MR2304922} and \cite[Definition 2.11]{bep}).

Let $\mathcal{A}^* \setminus \{ \epsilon \}$ denote the set of non-empty finite words over $\mathcal{A}$. Let $\{ S_a : a \in \mathcal{A} \}$ be the generators of $\mathcal{O}_\Lambda$ (respectively $\mathcal{O}_{\Lambda^*}$). For $u=u_1u_2\dots u_m\in\mathcal{A}^* \setminus \{ \epsilon \}$, let $S_u:=S_{u_1}S_{u_2}\dots S_{u_m}$.

\begin{corollary}[{cf. \cite[Theorem 4.9]{MR1646513} and \cite[Lemma 2.7]{MR1869072}}]
Let $\Lambda$ be a (two-sided) shift space over a finite alphabet $\mathcal{A}$. Let $\left( 1-\Phi \right) :\spa_\Z\{\cha{A} : A \in \mathcal{E}_{\Lambda^*}^{0,-}\}\to \spa_\Z\{\cha{A} : A\in\mathcal{E}_{\Lambda^*}^{0,-} \}$ be the linear map given by
\begin{equation*}
\left( 1-\Phi \right) (\cha{A})=\cha{A}-\sum_{a\in\mathcal{A}}\cha{r(A,a)} \quad A\in\mathcal{E}_{\Lambda^*}^{0,-}.
\end{equation*}
Then $K_1( \mathcal{O}_{\Lambda^*})$ is isomorphic to $\ker( 1 - \Phi )$ and there exists an isomorphism from $K_0(\mathcal{O}_{\Lambda^*})$ to $\coker(1 - \Phi)$ which maps $[ S_u^* S_u ]_0$ to $\cha{r(u)} + \im(1-\Phi)$ for each $u \in \mathcal{A}^* \setminus \{ \epsilon \}$.

If $\Lambda$ satisfies condition (I) of \cite{MR1691469}, then we have similar formulas for $K_0 ( \mathcal{O}_\Lambda )$ and $K_1 ( \mathcal{O}_\Lambda )$.
\end{corollary}

\begin{proof}
The labelled spaces $( E_{\Lambda} , \mathcal{L}_{\Lambda} , \mathcal{E}_{\Lambda}^{0,-} )$ and $( E_{\Lambda^*} , \mathcal{L}_{\Lambda^*} , \mathcal{E}_{\Lambda^*}^{0,-} )$ are weakly left-resolving and normal. Since $\mathcal{E}_{\Lambda}^{0,-}=(\mathcal{E}_{\Lambda}^{0,-})_J$ and $\mathcal{E}_{\Lambda^*}^{0,-}=(\mathcal{E}_{\Lambda^*}^{0,-} )_J$, the result follows from \cite[Theorem 6.3]{MR2304922} and Theorem \ref{theorem:ktheory}.
\end{proof}

\subsection*{Example 4 -- Cuntz-Pimsner algebras associated with subshifts}

Let $\OSS$ be a one-sided shift space over a finite alphabet $\mathcal{A}$. In \cite{TMC7} a $C^*$-algebra was associated with $\OSS$. An alternative construction of this $C^*$-algebra is given in \cite{CS} and its relationship with the algebras $\mathcal{O}_\Lambda$ and $\mathcal{O}_{\Lambda^*}$ is explored. We will denote this $C^*$-algebra by $C^*( \OSS )$ (it is denoted by $\mathcal{O}_\OSS$ in \cite{TMC7} and by $\mathcal{D}_\OSS\mspace{-2mu} \rtimes_{\alpha , \mathcal{L}} \N$ in \cite{CS}).

We show how one may realise $C^*( \OSS )$ as the $C^*$-algebra of a labelled space. We let $E_\OSS$ be the directed graph $(E_\OSS^0 , E^1_\OSS , r , s)$ where $E_\OSS^0 = \OSS$, $E_\OSS^1 = \{ (x,a,y) \in \OSS \times \mathcal{A} \times \OSS : x=ay\}$ and $r,s : E_\OSS^1\to E_\OSS^0$ are defined by $s(x,a,y) = x$ and $r(x,a,y)=y$. We let $\labe_\OSS : E_\OSS^1 \to \mathcal{A}$ be the labeling given by $\labe_\OSS (x,a,y)=a$.

For $u,v \in \labe_\OSS^\#(E_\OSS)$, let $C(u,v) = \{ vx \in \OSS : ux \in \OSS \}$. Let $\bool_\OSS$ be the Boolean algebra generated by $\{ C(u,v) : u,v \in \labe_\OSS^\#(E_\OSS) \}$. That is, $\bool_\OSS$ is the smallest subset of $2^{\OSS}$ which is closed under finite intersections, finite unions, relative complements and which contains $C(u,v)$ for every $u,v \in \labe_\OSS^\#(E_\OSS)$.

\begin{lemma} \label{lem:exisnice}
With the above definitions $(E_\OSS , \labe_\OSS , \bool_\OSS )$ is a normal weakly left-resolving labelled space.
\end{lemma}

\begin{proof}
We will first prove that $\bool_\OSS$ is accommodating for $(E_\OSS , \labe_\OSS )$. Let $\alpha \in \labe_\OSS^* ( E_\OSS )$. Then we have that  $r ( \alpha ) = C(\alpha,\epsilon)$, from which it follows that $\bool_\OSS$ contains $r ( \alpha )$. It remains to show that $\bool_\OSS$ is closed under relative ranges.

To check that $\bool_\OSS$ is closed under relative ranges it suffices to check that $r ( C(u,v) , a ) \in \bool_\OSS$ for all $u , v \in \labe_\OSS^\#(E_\OSS )$ and $a \in \mathcal{A}$. If $u,v \in \labe_\OSS^\#(E_\OSS)$ are such that $v = v_1 v_2 \dots v_n \ne \epsilon$, then we claim that
\begin{equation} \label{eq:relrange}
r ( C(u,v) , a) = \begin{cases} C ( a , \epsilon ) \cap C ( u , v_2 v_3 \dots v_n) & \text{ if } a=v_1 \\
\emptyset & \text{ otherwise.}
\end{cases}
\end{equation}
To see this observe that every vertex in $C(u,v)$ emits exactly one edge, labelled $v_1$, and so $r ( C(u,v) , a ) =\emptyset$ if $v_1\ne a$. If $v_1=a$ and $x\in C(a,\emptyword)\cap C(u,v_2v_3\dots v_n)$, then $ax\in C(u,v)$ and so $x\in r ( C(u,v) , a )$. If $v_1=a$ and $x\in r ( C(u,v) , a )$, then $ax\in C(u,v)$ and so $x\in C(a,\emptyword)\cap C(u,v_2v_3\dots v_n)$. Thus Equation \eqref{eq:relrange} holds. If $u\in \labe_\OSS^\#(E_\OSS)$ and $a \in \mathcal{A}$, then it is easy to check that $r(C(u,\emptyword),a)=C(ua,\emptyword)$. It follows that $\bool_\OSS$ is accommodating for $(E_\OSS,\labe_\OSS)$. Moreover, $\bool_\OSS$ is non-degenerate by definition.  

Since a vertex $x \in E_\OSS^0$ can only receive an edge labelled $a$ from the vertex $ax$ we see that $(E_\OSS,\labe_\OSS)$ is left-resolving. Thus $(E_\OSS , \labe_\OSS , \bool_\OSS )$ is weakly left-resolving and normal.
\end{proof}

As in \cite{TMC7}, we let $\DX$ denote the $C^*$-algebra generated by $\{ \cha{C( u , v )} : u,v \in \labe_\OSS^*(E_\OSS) \}$. We then have that $\DX = \spc \{ \cha{A} :  A \in \bool_\OSS \} = \Aalg[E_\OSS, \labe_\OSS , \bool_\OSS]$.

According to \cite[Remark 7.3]{TMC7}, $C^* ( \OSS )$ is the universal $C^*$-algebra generated by a family $\{ S_a :  a\in\mathcal{A} \}$ of partial isometries such that there exists a $*$-homomorphism from $\iota_\OSS : \DX \to C^*( \OSS )$ given by
\begin{equation} \label{eq:iotadef}
\iota_\OSS : \cha{C(u,v)} \mapsto ( S_{v_1} S_{v_2} \dots  S_{v_m} )( S_{u_1} S_{u_2} \dots S_{u_n})^*( S_{u_1} S_{u_2} \dots S_{u_n} )( S_{v_1} S_{v_2} \dots S_{v_m} )^*
\end{equation}
where $u = u_1 u_2\dots u_n$ and $v= v_1 v_2 \dots v_m \in \labe_\OSS^\#(E_\OSS)$. Note that if $u$ (respectively $v$) is the empty word, then the product $(S_{u_1} S_{u_2} \dots S_{u_n})$ (respectively $( S_{v_1} S_{v_2}\dots S_{v_m} )$) is the identity of $C^*( \OSS )$. If $u = u_1 u_2\dots u_n \in \labe_\OSS^*(E_\OSS) $, then we will denote the product $S_{u_1} S_{u_2} \dots S_{u_n}$ by $S_u$. If $u = \epsilon$, then we let $S_u$ denote the unit of $C^*(\OSS)$.

\begin{lemma} \label{lemma:conjsa}
Let $\iota_\OSS : \DX \to C^* ( \OSS )$ be defined in \eqref{eq:iotadef} above. Then for every $x \in \DX $ and $a \in \mathcal{A}$ we have $\iota_\OSS( \phi_a (x) ) = S_a^* \iota_\OSS( x) S_a$ where $\phi_a : \DX \to \DX$ is given by $\phi_a ( \chi_{A} ) = \chi_{r(A,a)}$ for $A \in \bool_\OSS$ (see Lemma \ref{lemma:phia}).
\end{lemma}

\begin{proof}
Let $u,v \in \labe_\OSS^\#(E_\OSS)$ with $v=v_1 v_2 \dots v_m \ne \epsilon$. If $a=v_1$ then
\begin{equation*}
\begin{split}
S_a^*\iota_\OSS(\cha{C(u,v)}) S_a & = S_a^* S_v S_u^* S_u S_v^* S_a
= S_a^* S_a S_{v_2} S_{v_3} \dots S_{v_m} S_u^* S_u ( S_{v_2} S_{v_3} \dots S_{v_m})^* \\
&= \iota_\OSS(\cha{C(a,\epsilon)}\cha{C(u,v_2v_3\dots v_m)}),
\end{split}
\end{equation*}
otherwise it is zero. Thus $S_a^*\iota_\OSS(\cha{C(u,v)}) S_a = \iota_\OSS(\cha{r(C(u,v),a)}) = \iota_\OSS(\phi_a ( \cha{C(u,v)}) )$ by Equation \eqref{eq:relrange}. One can similarly show that if $u \in \labe_\OSS^\#(E_\OSS)$, then $S_a^* \iota_\OSS(\cha{C(u,\epsilon)}) S_a = \iota_\OSS(\phi_a ( \cha{C(u,\epsilon)}) )$. Since $\DX$ is generated by $\{\cha{C(u,v)} : u,v \in \labe_\OSS^\#(E_\OSS) \}$, it follows that $\iota_\OSS( \phi_a (x) ) = S_a^* \iota_\OSS(x) S_a$ for every $x \in \DX$.
\end{proof}

\begin{proposition} \label{proposition:oss}
Let $\OSS$ be a one-sided shift space over a finite alphabet $\mathcal{A}$. Let $\{p_A, s_a : A\in \bool_\OSS, a\in\mathcal{A} \}$ and $\{S_a : a\in\mathcal{A} \}$ be the canonical generators of $C^*( E_\OSS , \labe_\OSS , \bool_\OSS )$ and $C^*( \OSS )$, respectively. Then the map $S_a \mapsto s_a$ for $a\in\mathcal{A}$ induces an isomorphism from $C^*( \OSS )$ to $C^*( E_\OSS , \labe_\OSS , \bool_\OSS )$.
\end{proposition}

\begin{proof}
Let $\iota_\OSS$ denote (the unique) $*$-homomorphism from $\DX$ to $C^*( \OSS )$ mapping $\cha{C(u,v)}$ to
\begin{equation*}
( S_{v_1} S_{v_2} \dots  S_{v_m} )( S_{u_1} S_{u_2} \dots S_{u_n})^*( S_{u_1} S_{u_2} \dots S_{u_n} )( S_{v_1} S_{v_2} \dots S_{v_m} )^*
\end{equation*}
for $u,v\in \labe_\OSS^\#(E_\OSS)$. We will show that $\{\iota_\OSS(\cha{A}),\ S_a : A \in \bool_\OSS,\ a\in\mathcal{A} \}$ is a representation of $( E_\OSS , \labe_\OSS , \bool_\OSS )$.
First, one checks that $\{\iota_\OSS(\cha{A}): \, A\in\bool_\OSS\}$ satisfies condition (i) of Definition \ref{lgdef}. Second, if $a \in \mathcal{A}$ and $A \in \bool_\OSS$, then it follows from Lemma \ref{lemma:conjsa} that
\begin{equation*}
\iota_\OSS (\cha{A}) S_a  = \iota_\OSS ( \cha{A} ) S_a S_a^* S_a
= S_a S_a^* \iota_\OSS (\cha{A}) S_a
= S_a \iota_\OSS ( \phi_a ( \cha{A }) )  = S_a \iota_\OSS ( \cha{r(A,a)} )
\end{equation*}
which shows that $\{\iota_\OSS(\cha{A}),\ S_a : A\in \bool_\OSS,\ a\in\mathcal{A} \}$ satisfies condition (ii) of Definition \ref{lgdef}. If $a \in \mathcal{A}$, then $S_a^* S_a  = \iota_\OSS (\cha{C(a,\epsilon)})  = \iota_\OSS (\cha{r(a)})$. If $a, b \in \mathcal{A}$ and $a \ne b$, then $S_a^*S_b= S_a^* \iota_\OSS( \cha{C(\epsilon,a)} \cha{C ( \epsilon, b )} )S_b = 0$. Thus $\{\iota_\OSS(\cha{A}),\ S_a : A\in \bool_\OSS,\ a\in\mathcal{A} \}$ satisfies condition (iii) of Definition \ref{lgdef}. We have that $\bigcup_{a\in\mathcal{A}}C( \epsilon,a ) = E_\OSS^0$ and $C(\epsilon,a)\cap C(\epsilon,b)=\emptyset$ for $a ,b \in \mathcal{A}$ with $a \ne b$ and so $\cha{E^0_\OSS} = \sum_{a \in \labe_\OSS ( E_{\OSS}^1 ) } \cha{C(\epsilon,a)}$ in $\DX$. Since $\iota_\OSS ( \cha{E^0_\OSS} ) = 1_{C^* ( \OSS )}$ it follows from Lemma \ref{lemma:conjsa} that if $A\in\bool_\OSS$, then
\begin{equation*}
\begin{split}
\iota_\OSS(\cha{A} )
&= \sum_{a \in \labe_\OSS ( E_{\OSS}^1 ) } \iota_\OSS (\cha{C(\epsilon,a)} \cha{A} \cha{ C(\epsilon,a)} )
= \sum_{a \in \labe_\OSS ( AE_{\OSS}^1 ) } S_a S_a^* \iota_\OSS( \cha{A}) S_a S_a^*  \\
&= \sum_{a \in \labe_\OSS ( AE_{\OSS}^1 ) } S_a \iota_\OSS ( \phi_a ( \cha{A} ) )  S_a^*
= \sum_{a \in \labe_\OSS ( AE_{\OSS}^1 ) } S_a  \iota_\OSS ( \cha{r(A,a)} )  S_a^*
\end{split}
\end{equation*}
which shows that $\{\iota_\OSS(\cha{A}),\ S_a : A\in \bool_\OSS,\ a\in\mathcal{A} \}$ satisfies condition (iv) of Definition \ref{lgdef}. Hence $\{\iota_\OSS(\cha{A}),\ S_a : A\in \bool_\OSS,\ a\in\mathcal{A} \}$ is a representation of $( E_\OSS , \labe_\OSS , \bool_\OSS  )$. It follows from the universal property of $C^* ( E_\OSS , \labe_\OSS , \bool_\OSS  )$ that there is a $*$-homomorphism $\pi_{s,p} : C^* ( E_\OSS , \labe_\OSS , \bool_\OSS  ) \to C^* ( \OSS )$ which sends $p_A$ to $\cha{A}$ and $s_a$ to $S_a$ for all $A \in \bool_\OSS$ and $a \in \mathcal{A}$.

It follows from Lemma \ref{lemma:map} that there exists a $*$-homomorphism from $\DX = \Aalg[E_\OSS,\labe_\OSS,\bool_\OSS]$ to $C^* ( E_\OSS , \labe_\OSS , \bool_\OSS  )$ which maps $\cha{A}$ to $p_A$ for every $A \in \bool_\OSS$. In particular, if $u,v \in \labe_\OSS^\#( E_\OSS )$, then $\phi( \cha{C(u,v)} ) = p_{C(u,v)}$. Repeated applications of conditions (iii) and (iv) of Definition \ref{lgdef} show that if $u \in \labe_\OSS^\#( E_\OSS )$  then $s_u^* s_u = p_{r(u)} = p_{C ( u , \epsilon)}$ (if $u=\emptyword$, then $r(u)=C(u,\emptyword)=\OSS$). For $v = v_1 v_2 \dots v_m \in \labe_\OSS^*( E_\OSS )$, condition (iv) of Definition \ref{lgdef} implies that
\begin{equation*}
\begin{split}
p_{C(u,v)} & = \!\! \sum_{a \in \mathcal{L}_{\OSS} ( C(u,v) E_{\OSS}^1 )} \!\!\!\!\! s_a p_{r(C(u,v),a)} s_a^*
= s_{v_1} p_{C ( u , v_2 v_3\dots v_m)}s_{v_1}^* \text{ by Equation \eqref{eq:relrange}} \\
&= s_{v_1} \left( \sum_{a \in \mathcal{L}_{\OSS} ( C(u, v_2 v_3 \dots v_m) E_{\OSS}^1 )} \!\!\!\!\! s_a p_{r(C(u , v_2 v_3 \dots v_m),a)} s_a^* \right) s_{v_1}^*
= s_{v_1} s_{v_2} p_{C(u,v_3\dots v_m)} s_{v_2}^* s_{v_1}^* \\
&= \dots = s_{v_1} s_{v_2} \dots s_{v_m} p_{C(u,\epsilon)} s_{v_m}^* \dots s_{v_2}^* s_{v_1}^* = s_v s_u^* s_u s_v^*.
\end{split}
\end{equation*}
Thus, it follows from the universal property of $C^*( \OSS )$ that there is a $*$-homomorphism $\pi_{S,P} : C^*( \OSS ) \to C^* ( E_\OSS , \labe_\OSS , \bool_\OSS  )$ which sends  $S_a$ to $s_a$ and $\iota_\OSS(\cha{A})$ to $p_A$ for all $a \in \mathcal{A}$  and  $A \in \bool_\OSS$.

Since $\pi_{s,p}$ and $\pi_{S,P}$ are evidently inverses, the result follows.
\end{proof}

\begin{corollary}[{cf. \cite[Theorem 1.1]{CS2}}] \label{cor:k-oss}
Let $\OSS$ be a one-sided shift space over a finite alphabet $\mathcal{A}$. Let $(1 - \Phi ) : \spa_\Z \{ \cha{A} : A \in \bool_\OSS \} \to \spa_\Z \{ \cha{A}: A \in \bool_\OSS \}$ be the linear map given by
\begin{equation*}
( 1 - \Phi ) ( \cha{A} ) = \cha{A} - \sum_{a \in \mathcal{L}_\OSS ( A E_\OSS^1 )} \cha{r(A,a)}, \quad A\in\bool_\OSS.
\end{equation*}
Then $K_1(C^*( \OSS ) )$ is isomorphic to $\ker( 1 - \Phi  )$ and there exists an isomorphism from $K_0 ( C^*( \OSS ) )$ to $\coker( 1 - \Phi )$ which maps $[ S_u^* S_u ]_0$ to $\cha{r(u)} + \im(1 - \Phi )$ for each $u \in \labe_\OSS^*( E_\OSS )$.
\end{corollary}

\begin{proof}
Notice that $(\bool_\OSS)_J=\bool_\OSS$. The result now follows from Theorem \ref{theorem:ktheory} and Proposition \ref{proposition:oss}.
\end{proof}

\section{Computing K-Theory for $C^* (E , \mathcal{L} , \mathcal{E}^{0,-})$} \label{sec:howdismantleanatombomb}

Recall from \cite{bp3} that a labelled space $( E , \mathcal{L}  , \mathcal{B}  )$ is \emph{set-finite} if $\labe(A E^1)$ is finite for all $A \in {\mathcal B}$, and \emph{receiver set-finite} if for all $A \in \mathcal{B} $ and all $\ell \ge 1$ the set $\{ \mathcal{L} (\lambda) : \lambda \in E^\ell ,\ r ( \lambda ) \in A \}$ is finite.

In this section we give a practical method for computing the $K$-theory of labelled graph algebras of the form $C^* ( E , \mathcal{L} , \mathcal{E}^{0,-} )$ for left-resolving labelled graphs $(E,\mathcal{L})$ with no sources and sinks such that $(E,\mathcal{L},\mathcal{E}^{0,-})$ is a set-finite and receiver set-finite labelled space. We use Theorem \ref{theorem:ktheory} to give a description of the $K$-theory of $C^*(E,\mathcal{L},\mathcal{E}^{0,-})$ as an inductive limit in analogy with the computations of \cite{MR1911208}.

\subsubsection*{\bf Standing Assumption:}

Throughout this section $(E,\mathcal{L})$ shall be a left-resolving  labelled graph with no sources and sinks such that $(E, \mathcal{L},\mathcal{E}^{0,-})$ is a set-finite and receiver set-finite (weakly left-resolving normal) labelled space.

Before we begin, let us recall some notation from \cite{bp3}: For $v \in E^0$ and $\ell \ge 1$ let
\begin{equation*}
\Lambda_\ell (v) = \{ \alpha \in \mathcal{L} ( E^k ) :k \le \ell ,\ v \in r ( \alpha ) \}
\end{equation*}
denote the set of labelled paths of length at most $\ell$ which have a representative terminating at $v$. Define the relation $\thicksim_\ell$ on $E^0$ by $v \thicksim_\ell w$ if and only if $\Lambda_\ell (v) = \Lambda_\ell (w)$; hence $v \thicksim_\ell w$ if $v$ and $w$ receive exactly the same labelled paths of length at most $\ell$. Evidently $\thicksim_\ell$ is an equivalence relation and we use $[v]_\ell$ to denote the equivalence class of $v \in E^0$. We call the $[v]_\ell$ \emph{generalised vertices} as they play the same role in labelled spaces as vertices in a directed graph.

Set $\Omega_\ell = E^0 / \thicksim_\ell =\{ [v]_\ell : v\in E^0\}$ and let
\begin{equation*}
\Omega = \bigcup_{\ell=1}^\infty\Omega_\ell.
\end{equation*}
It follows from \cite[Proposition 2.4]{bp3} that if $v \in E^0$ and $\ell \ge 1$, then
\begin{equation} \label{eq:decomp}
[v]_\ell = \bigcup_{i=1}^n [ w_i ]_{\ell + 1} \text{ for some } w_1,\dots, w_n \in [v]_\ell .
\end{equation}
We then have that
\begin{equation} \label{eq:relrange2}
r ( [v]_\ell , a ) = \bigcup_{j=1}^m r ( [w_j]_{\ell+1} , a ) \text{ for all } a \in \mathcal{A}.
\end{equation}

It also follows from \cite[Proposition 2.4]{bp3} that $\mathcal{E}^{0,-}$ is the smallest subset of $2^{E^0}$ that contains $\Omega$ and is closed under finite unions and relative complements.

We let ${\Z}(\Omega)$ denote the subgroup $\spa_{\Z} \{ \chi_{[v]_\ell} \;:\; [v]_\ell \in \Omega \}$ of the group of functions from $E^0$ to $\Z$.

\begin{corollary}
Let $(1 - \Phi ) :  {\Z}(\Omega) \to  {\Z}(\Omega)$ be the linear map defined by
\begin{equation*}
(1 - \Phi)(\chi_{[v]_\ell}) = \chi_{[v]_\ell} - \sum_{a \in \labe ([v]_\ell E^1) } \chi_{r([v]_\ell,a)} \text{ for } [v]_\ell \in \Omega .
\end{equation*}
Then
\begin{eqnarray*}
K_1(C^*(E,{\mathcal L},{\mathcal E}^{0,-})) &\cong& \operatorname{ker}(1 - \Phi) \mbox{ and }\\
K_0(C^*(E,{\mathcal L},{\mathcal E}^{0,-}))  &\cong& \operatorname{coker}(1 - \Phi)
\end{eqnarray*}
via $[p_{[v]_\ell}]_0 \mapsto \chi_{[v]_\ell} + \operatorname{Im}(1 - \Phi)$ for $[v]_\ell \in \Omega$.
\end{corollary}

\begin{proof}
Since $\spa_{\Z} \{ \chi_A : A\in \mathcal{E}^{0,-} \}=\spa_{\Z} \{ \chi_A : A\in \mathcal{E}^{0,-}_J \}={\Z}(\Omega)$, the result follows from Theorem \ref{theorem:ktheory}.
\end{proof}

In Theorem \ref{thm:howtocompute} we shall show how to compute the kernel and cokernel of $(1 - \Phi)$ in terms of a certain inductive limit. For $\ell \ge 1$ define $\Z(\Omega_\ell) = \spa_{\Z} \{ \chi_{[v]_\ell} : [v]_\ell \in \Omega_\ell \}$.  Note that since for every $v \in E^0$ and $\ell \ge 1$ there are vertices $w_1, \dots, w_m \in [v]_\ell$ such that $[v]_\ell = \bigcup_{j=1}^m [w_j]_{\ell + 1}$ we have a linear map $i_\ell : \Z(\Omega_\ell) \to \Z(\Omega_{l + 1})$ defined by \begin{equation*} i_\ell(\chi_{[v]_\ell}) = \sum_{j=1}^{m} \chi_{[w_j]_{\ell + 1}}.
\end{equation*}
Therefore we have an inductive system $\displaystyle{ \lim_{\longrightarrow}  (\Z (\Omega_\ell),i_\ell) }$. Note that since, for fixed $\ell \ge 1$, the sets $[v]_\ell$ are disjoint, the functions $\cha{[v]_\ell}$ form a basis for $\Z(\Omega_\ell)$.

\begin{proposition}
There is an isomorphism
\begin{equation*}
\Psi_\infty : \lim_{\longrightarrow} (\Z (\Omega_\ell),i_\ell) \to \Z (\Omega)
\end{equation*}
such that $\Psi_\infty (\cha{[v]_\ell}) = \cha{[v]_\ell}$ for all $[v]_\ell \in \Omega$.
\end{proposition}

\begin{proof}
For each $\ell \ge 1$ define  $\Psi_\ell : {\Z}(\Omega_\ell) \to {\Z}(\Omega)$ by $\Psi_\ell(\chi_{[v]_\ell}) = \chi_{[v]_\ell}$.  One checks that $\Psi_\ell \circ (i_\ell \circ \dots \circ i_k) = \Psi_k$ for every $\ell < k$ and so there is a map $\Psi_\infty :  \displaystyle{\lim_{\longrightarrow} ( \Z (\Omega_\ell),i_\ell) \to {\Z}(\Omega)}$.  Since $\Z (\Omega) = \bigcup_{\ell \ge 1} \Psi_\ell({\Z}(\Omega_\ell))$ and each $\Psi_\ell$ is injective, it follows that $\Psi_\infty$ is an isomorphism and our result is established.
\end{proof}

Equations \eqref{eq:decomp} and \eqref{eq:relrange2} allow us to make the following definitions.

\begin{definition} \label{def:oneminusphielldef}
For $\ell \ge 1$ define a linear map $(1 - \Phi)_\ell : {\Z}(\Omega_\ell) \to {\Z}(\Omega_{\ell + 1})$ by
\begin{equation*}
(1 - \Phi)_\ell \left( \chi_{[v]_\ell} \right) = i_\ell \left(\chi_{[v]_\ell}\right) -
\sum_{a \in \labe ([v]_\ell E^1)} \chi_{r([v]_\ell,a)},
\end{equation*}
and define a linear map $(1-\Phi) : \Z ( \Omega ) \to \Z ( \Omega )$ by
\begin{equation*}
(1 - \Phi ) \left( \chi_{[v]_\ell} \right) = \chi_{[v]_\ell} - \sum_{a \in \labe ([v]_\ell E^1)} \chi_{r([v]_\ell,a)}.
\end{equation*}
\end{definition}

\begin{lemma} \label{lemma:Phiinfinite}
For all $\ell \ge 1$, we have $(1 - \Phi)_{\ell +1} \circ i_\ell = i_{\ell+1} \circ (1 - \Phi)_\ell$.  Therefore the $(1 - \Phi)_\ell$'s induce a map $\displaystyle{(1 - \Phi)_\infty:   \lim_{\longrightarrow} ( \Z(\Omega_\ell ) ,i_\ell) \to  \lim_{\longrightarrow} ( \Z(\Omega_\ell ),i_\ell)}$ which satisfies $(1 - \Phi) \circ \Psi_\infty = \Psi_\infty \circ (1 - \Phi)_\infty$.
\end{lemma}

\begin{proof}
Let $\ell \ge 1$.  Then for $[v]_\ell \in \Omega_\ell$ with $[v]_\ell = \bigcup_{j=1}^m [w_j]_{\ell + 1}$ we have
\begin{align*}
(1 - \Phi)_{\ell + 1}  \left(i_\ell \left(\chi_{[v]_\ell}\right) \right )
&= (1 - \Phi)_{\ell +1} \left(\sum_{j=1}^m \chi_{[w_j]_{\ell + 1}} \right) \\
&= \sum_{j=1}^m \left(i_{\ell + 1} \left( \chi_{[w_j]_{\ell + 1}} \right) - \sum_{a \in \labe ([w_j]_{\ell + 1} E^1) } \chi_{r([w_j]_{\ell + 1},a)}\right).
\end{align*}
Also,
\begin{align*}
i_{\ell + 1} \left((1 - \Phi)_\ell \left(\chi_{[v]_\ell}\right)\right)
&= i_{\ell + 1}\left( i_\ell \left( \chi_{[v]_\ell} \right) - \sum_{a \in \labe ([v]_\ell E^1 )}\chi_{r([v]_\ell,a)}\right) \\
&= i_{\ell + 1} \left( \sum_{j=1}^m \chi_{[w_j]_{\ell + 1}} - \sum_{j=1}^m \sum_{a \in \labe( [w_j]_{\ell + 1} E^1)} \chi_{r([w_j]_{\ell + 1},a)} \right) \intertext{ since $\chi_{r([v]_\ell,a)} = \sum_{j=1}^m \chi_{r([w_j]_{\ell + 1},a)}$,}
&= \sum_{j=1}^m \left( i_{\ell + 1} \left( \chi_{[w_j]_{\ell + 1}} \right) - \sum_{a \in \labe ([w_j]_{\ell + 1} E^1) } \chi_{r([w_j]_{\ell + 1},a)} \right)
\end{align*}
as $i_{\ell +1}$ is really just the identity map on $\Omega_{\ell + 1}$ and so $(1 - \Phi)_{\ell +1} \circ i_\ell = i_{\ell+1} \circ (1 - \Phi)_\ell$.

Also for $\ell \ge 1$ we have
\begin{equation*}
(1 - \Phi) (\Psi_\infty (\chi_{[v]_\ell})) = (1 - \Phi) (\chi_{[v]_\ell})
= i_\ell(\chi_{[v]_\ell}) - \sum_{a \in  \labe ([v]_\ell E^1) } \chi_{r([v]_\ell,a)}
\end{equation*}
and
\begin{eqnarray*}
\Psi_\infty( (1 - \Phi)_\infty (\chi_{[v]_\ell})) &=& \Psi_\infty ( (1 - \Phi)_\ell(\chi_{[v]_\ell})) = \Psi_\infty \left( i_\ell(\chi_{[v]_\ell})
- \sum_{a \in \labe ([v]_\ell E^1) } \chi_{r([v]_\ell,a)}
\right) \\
&=& i_\ell(\chi_{[v]_\ell}) - \sum_{a \in \labe ([v]_\ell E^1 )} \chi_{r([v]_\ell,a)}
\end{eqnarray*}
and so $(1 - \Phi) \circ \Psi_\infty = \Psi_\infty \circ (1 - \Phi)_\infty$.
\end{proof}

We have established the following commuting diagram:
\begin{equation*}
\beginpicture
\setcoordinatesystem units <1cm,1cm>

\setplotarea x from -5 to 7, y from  -2  to 2

\put{$\Z(\Omega_1)$}[l] at  -4 1

\put{$\Z(\Omega_2)$}[l] at  -2 1

\put{$\Z(\Omega_2)$}[l] at  0 1

\put{$\dots$}[l] at 2 1

\put{$\displaystyle{\lim_{\longrightarrow} (\Z(\Omega_\ell),i_\ell)}$}[l] at 3 1

\put{$\Z(\Omega)$}[l] at 6.5 1

\put{$\Z(\Omega_1)$}[l] at  -4 -1

\put{$\Z(\Omega_2)$}[l] at  -2 -1

\put{$\Z(\Omega_2)$}[l] at  0 -1

\put{$\dots$}[l] at 2 -1

\put{$\displaystyle{\lim_{\longrightarrow} (\Z(\Omega_\ell),i_\ell)}$}[l] at 3 -1

\put{$\Z(\Omega)$}[l] at 6.5 -1

\put{$i_1$}[b] at -2.5 1.2

\put{$i_2$}[b] at -0.5 1.2

\put{$i_3$}[b] at 1.5 1.2

\put{$\Psi_\infty$}[b] at 6 1.2

\put{$i_1$}[t] at -2.5 -1.2

\put{$i_2$}[t] at -0.5 -1.2

\put{$i_3$}[t] at 1.5 -1.2

\put{$\Psi_\infty$}[t] at 6 -1.2

\put{$(1 - \Phi)_\infty$}[l] at 4.2 0

\put{$1 - \Phi$}[l] at 7.2 0

\put{$(1-\Phi)_1$}[l] at   -2.8     0.25

\put{$(1-\Phi)_2$}[l] at   -0.8     0.25

\arrow <0.25cm> [0.1,0.3] from -2.9 1 to -2.1 1

\arrow <0.25cm> [0.1,0.3] from -0.9 1 to -0.1 1

\arrow <0.25cm> [0.1,0.3] from 1.1 1 to 1.9 1

\arrow <0.25cm> [0.1,0.3] from -2.9 1 to -2.1 1

\arrow <0.25cm> [0.1,0.3] from 5.5 1 to 6.4 1

\arrow <0.25cm> [0.1,0.3] from -2.9 -1 to -2.1 -1

\arrow <0.25cm> [0.1,0.3] from -0.9 -1 to -0.1 -1

\arrow <0.25cm> [0.1,0.3] from 1.1 -1 to 1.9 -1

\arrow <0.25cm> [0.1,0.3] from -2.9 -1 to -2.1 -1

\arrow <0.25cm> [0.1,0.3] from 5.5 -1 to 6.4 -1

\arrow <0.25cm> [0.1,0.3] from 4 0.6 to 4 -0.6

\arrow <0.25cm> [0.1,0.3] from 7 0.6 to 7 -0.6

\arrow <0.25cm> [0.1,0.3] from -3.6 0.6 to -1.8 -0.6

\arrow <0.25cm> [0.1,0.3] from -1.6 0.6 to 0.2 -0.6
\endpicture
\end{equation*}

Let $\ell \ge 1$. It is easy to check that $i_{\ell} (\operatorname{ker} ( 1 - \Phi )_\ell) \subseteq \operatorname{ker} ( 1 - \Phi )_{\ell +1}$ and $i_{\ell} (\operatorname{Im} ( 1 - \Phi )_\ell) \subseteq \operatorname{Im} ( 1 - \Phi )_{\ell +1}$. It follows that $i_\ell : \Z ( \Omega_\ell ) \to \Z ( \Omega_{\ell+1} )$ induces maps $(i_\ell )_* : \operatorname{ker} ( 1 - \Phi )_\ell \to \operatorname{ker} ( 1 - \Phi )_{\ell+1}$ and $\tilde{i}_{\ell+1} : \operatorname{coker} ( 1 - \Phi )_\ell \to \operatorname{coker} ( 1 - \Phi )_{\ell +1}$.

\begin{theorem} \label{thm:howtocompute}
With the above notation we have
\begin{align}
\operatorname{ker} (1 - \Phi) &\cong \lim_{\longrightarrow} \left( \operatorname{ker} (1 - \Phi)_\ell , (i_\ell)_* \right) \label{eq:ker}, \\
\operatorname{coker}(1 - \Phi) &\cong \lim_{\longrightarrow} \left( \operatorname{coker}(1 - \Phi)_\ell,\tilde{i_\ell} \right) \label{eq:coker}.
\end{align}
\end{theorem}

\begin{proof}
By definition of $(1 - \Phi)_\infty$ we have
\begin{equation*}
\operatorname{ker} (1 - \Phi)_\infty \cong \lim_{\longrightarrow} \left( \operatorname{ker} (1 - \Phi)_\ell, (i_\ell)_* \right).
\end{equation*}
It follows from Lemma \ref{lemma:Phiinfinite} that $\operatorname{ker} (1 - \Phi) \cong \operatorname{ker} (1 - \Phi)_\infty$ and Equation \eqref{eq:ker} follows. We claim that
\begin{equation*}
\operatorname{coker}(1 - \Phi)_\infty \cong \lim_{\longrightarrow} \left( \operatorname{coker}(1 - \Phi)_\ell,\tilde{i_\ell} \right).
\end{equation*}
Define $\eta_\ell : \operatorname{coker}  (1 - \Phi)_\ell \to \operatorname{coker} (1 - \Phi)_\infty$ by
\begin{equation*}
\eta_\ell \left( \chi_{[v]_{\ell + 1}} + \operatorname{Im} (1 - \Phi)_\ell \right) =
\chi_{[v]_\ell} + \operatorname{Im} (1 - \Phi)_\infty.
\end{equation*}
Note that $\eta_\ell$ is well-defined since $\operatorname{Im} (1 - \Phi)_\ell \subseteq \operatorname{Im} (1 - \Phi)_\infty$ for all $\ell \ge 1$.  We claim that $\eta_{\ell + 1} \circ \tilde{i}_\ell = \eta_\ell$ for all $\ell \ge 1$.  Since
\begin{eqnarray*}
\eta_{\ell + 1} \tilde{i}_\ell \left( \chi_{[v]_{\ell + 1}} + \operatorname{Im} (1 - \Phi)_\ell \right)
&=& \eta_{\ell + 1} \left( i_{\ell + 1} \chi_{[v]_{\ell + 1}}  + \operatorname{Im} (1 - \Phi)_{\ell + 1} \right)  \\
&=& i_{\ell + 1} \left( \chi_{[v]_{\ell + 1}}\right) + \operatorname{Im} (1 - \Phi)_\infty  \\
&=&  \chi_{[v]_{\ell + 1}} + \operatorname{Im} (1 - \Phi)_\infty \\
&=& \eta_\ell (\chi_{[v]_{\ell + 1}} + \operatorname{Im} (1 - \Phi)_\ell) ,
\end{eqnarray*}
this establishes our claim. By the universal property of the inductive limit the $\eta_\ell$'s  induce a map $\displaystyle \eta_\infty : \lim_{\longrightarrow} \left( \operatorname{coker} (1 - \Phi)_\ell,\tilde{i}_\ell \right) \to \operatorname{coker} (1 - \Phi)_\infty$ which is injective since each $\eta_\ell$ is injective, and is surjective as $\operatorname{coker} (1 - \Phi)_\infty = \bigcup_{\ell \ge 1} \eta_\ell \left(\operatorname{coker} (1 - \Phi)_\ell \right)$. Hence $\eta_\infty$ is an isomorphism which establishes our claim. Finally, it follows from Lemma \ref{lemma:Phiinfinite} that $\operatorname{coker} (1 - \Phi)_\infty \cong \operatorname{coker} (1 - \Phi)$ and Equation \eqref{eq:coker} follows.
\end{proof}

\begin{remark}
Recall from \cite{bp3} that to a left-resolving normal labelled graph $(E,\mathcal{L})$ with no sources and sinks over a finite alphabet we may associate an essential symbolic matrix system $(M(E),I(E))$, which by \cite[Proposition 2.1]{m99} determines a unique $\lambda$-graph system $\mathfrak{L}_{M(E),I(E)}$.

By \cite[Proposition 3.6]{bp3} we know that $C^* ( E , \mathcal{L} , \mathcal{E}^{0,-} )$ is isomorphic to $\mathcal{O}_{\mathfrak{L}_{M(E),I(E)}}$. We should therefore expect some similarities between our computation of the $K$-theory of $C^* ( E , \mathcal{L} , \mathcal{E}^{0,-} )$ and the computation of the $K$-theory of $\mathcal{O}_{\mathfrak{L}_{M(E),I(E)}}$ outlined in \cite[Section 9]{m99} (see also \cite{MR1646513,MR1869072}). Indeed, this is the case.

To an essential symbolic matrix system $(M,I) = ( M_{\ell,\ell+1} , I_{\ell,\ell+1} )_{\ell \ge 1}$ we may associate a $\lambda$-graph system $\mathfrak{L}_{M,I}$ as in \cite[Section 2]{m99}. Following \cite[Section 9]{m99} we see that
\begin{equation*}
K_* ( \mathcal{O}_{\mathfrak{L}_{M,I}} ) = K_* ( M,I ) = \varinjlim ( K_*^\ell ( M,I ) , i_*^\ell )
\end{equation*}
for $*=0,1$ where
\begin{align*}
K_0^\ell (M,I) &= \Z^{m(\ell+1)} / ( I^t_{\ell,\ell+1} - M^t_{\ell,\ell+1} ) \Z^{m(\ell)} \\
K_1^\ell (M,I) &= \operatorname{ker} \, ( I^t_{\ell,\ell+1} - M^t_{\ell,\ell+1} ) \text{ in } \Z^{m(\ell)}
\end{align*}
and $i^\ell_* : K^\ell_* ( M,I) \to K^{\ell+1}_* (M,I)$ is induced by the map $I^t_{\ell,\ell+1} : \Z^{m(\ell)} \to \Z^{m(\ell+1)}$.

Under our identification of the labelled graph $(E, \mathcal{L} )$ with the essential symbolic matrix system $(M(E),I(E))$ we have $| \Omega_\ell | = m( \ell )$, and so we may identify the group $\Z ( \Omega_\ell )$ with $\Z^{m(\ell)}$ and the map $( 1 - \Phi )_\ell : \Z ( \Omega_\ell ) \to \Z ( \Omega_{\ell+1} )$ with the map $( I(E)^t_{\ell,\ell+1} - M(E)^t_{\ell,\ell+1} ) : \Z^{m(\ell)} \to \Z^{m ( \ell+1 )}$.  Hence $\operatorname{coker} ( 1 - \Phi )_\ell$ may be identified with $K_0^\ell ( M(E),I(E) )$ and $\operatorname{ker} ( 1 - \Phi )_\ell$ may be identified with  $K_1^\ell (M(E),I(E))$. Since the map $( i_\ell )_* : \operatorname{ker} ( 1 - \Phi )_\ell  \to \operatorname{ker} ( 1 - \Phi )_{\ell+1}$ corresponds to $i^\ell_0 : K_0^\ell ( M(E),I(E) ) \to K_0^{\ell+1} (M(E),I(E))$ and the map $\tilde{i}_\ell : \operatorname{coker} ( 1 - \Phi )_\ell \to \operatorname{coker} ( 1 - \Phi )_{\ell+1}$ corresponds to $i^\ell_1 : K_1^\ell (M(E),I(E)) \to K_1^{\ell+1} (M(E),I(E))$, we may identify $\operatorname{coker} ( 1 - \Phi )$ with $K_0 ( M(E), I(E) )$ and $\operatorname{ker} ( 1 - \Phi )$ with $K_1 ( M(E), I(E) )$.
\end{remark}

\section{Computations} \label{sec:computations}

In this section we compute the $K$-theory of certain weakly left-resolving normal  labelled graphs using the techniques outlined in Section \ref{sec:howdismantleanatombomb}.

\begin{example}
Let $( E_F , \mathcal{L}_F )$ be the left Fischer cover (see \cite[Example 2.15]{bep}) of the even shift. Then $\Omega_\ell = E_F^0$ for all $\ell \ge 1$ and so $\Z ( \Omega ) = \Z^2$ and the matrix of $(1-\Phi)$ is $\left(
\begin{smallmatrix}
    0 & -1 \\
    -1 & 1
\end{smallmatrix}
\right)$. Hence the $K$-theory of the labelled graph is the same as the $K$-theory of the underlying graph, that is $K_0 = K_1 = \{ 0 \}$. This should not be a surprise in the light of \cite[Theorem 6.6]{MR2304922}. We obtain similar conclusions for the left Krieger cover $(E_K,{\mathcal L}_K)$ of the even shift, though in this case $\Omega_\ell = E_K^0$ for $\ell \ge 2$.
\end{example}

\begin{example}
Now consider the labelled graph $(E,{\mathcal L})$ shown below
\begin{equation*}
\beginpicture
\setcoordinatesystem units <2cm,1cm>

\setplotarea x from -3.1 to 3, y from  0.3  to 1

\put{$\ldots$}[l] at -3.3 0

\put{$\bullet$} at -2 0

\put{$\bullet$} at -1 0

\put{$\bullet$} at 0 0

\put{$\bullet$} at 1 0

\put{$\bullet$} at 2 0

\put{$\bullet$} at -3 0

\put{$\bullet$} at 3 0

\put{$\ldots$}[l] at 3.1 0

\put{ $b$} at -2.5 0.4

\put{ $b$} at -1.5 0.4

\put{$b$} at -0.5 0.4

\put{$b$} at 0.5 0.4

\put{$b$} at 1.5 0.4

\put{ $b$} at 2.5 0.4

\put{ $c$} at -2.5 -0.7

\put{$c$} at -1.5 -0.7

\put{$c$} at -0.5 -0.7

\put{$c$} at 0.5 -0.7

\put{$c$} at 1.5 -0.7

\put{$c$} at 2.5 -0.7

\put{$a$}[b] at 0 1.2

\circulararc 360 degrees from 0 0 center at 0 0.5

\setquadratic

\plot -1.1 -0.1 -1.5 -0.5 -1.9 -0.1 /

\plot -2.1 -0.1 -2.5 -0.5 -2.9 -0.1 /

\plot -0.1 -0.1 -0.5 -0.5 -0.9 -0.1 /

\plot 0.1 -0.1 0.5 -0.5 0.9 -0.1 /

\plot 1.1 -0.1 1.5 -0.5 1.9 -0.1 /

\plot 2.1 -0.1 2.5 -0.5 2.9 -0.1 /

\arrow <0.25cm> [0.1,0.3] from 0.05 1 to -0.05 1.04

\arrow <0.25cm> [0.1,0.3] from -2.9 0 to -2.1 0

\arrow <0.25cm> [0.1,0.3] from -1.9 0 to -1.1 0

\arrow <0.25cm> [0.1,0.3] from -0.9 0 to -0.1 0

\arrow <0.25cm> [0.1,0.3] from 0.1 0 to 0.9 0

\arrow <0.25cm> [0.1,0.3] from 1.1 0 to 1.9 0

\arrow <0.25cm> [0.1,0.3] from 2.1 0 to 2.9 0

\arrow <0.25cm> [0.1,0.3] from -2.8 -0.27 to -2.9 -0.1

\arrow <0.25cm> [0.1,0.3] from -1.8 -0.27 to -1.9 -0.1

\arrow <0.25cm> [0.1,0.3] from -0.8 -0.27 to -0.9 -0.1

\arrow <0.25cm> [0.1,0.3] from 0.2 -0.27 to 0.1 -0.1

\arrow <0.25cm> [0.1,0.3] from 1.2 -0.27 to 1.1 -0.1

\arrow <0.25cm> [0.1,0.3] from 2.2 -0.27 to 2.1 -0.1

\endpicture
\end{equation*}
which was discussed in \cite[Section 7.2]{bp3} and \cite{kj}. We label the vertices of $E$ by the integers, where $0 = r(a)$.

Fix $\ell \ge 1$. We set $\text{rest}_\ell = \{ \pm \ell, \pm (\ell + 1) , \ldots \}$. Then $[n]_\ell =  \{n\}$ for $| n  | \le \ell -1$ and $[n]_\ell = \text{rest}_\ell$ for $| n | \ge \ell$. Hence, if we identify each vertex $n$ with the singleton $\{n\}$, then we have that
\begin{equation*}
\Omega_1 = \{ 0 ,  \text{rest}_1 \bigr\},
\Omega_2 = \{ 0 , \pm 1 , \text{rest}_2 \}, \ldots ,
\Omega_\ell = \{0 , \pm 1 , \pm 2 , \ldots , \pm (\ell -1) ,  \text{rest}_\ell \}, \dots .
\end{equation*}

By Definition \ref{def:oneminusphielldef} we have $(1-\Phi)_1 ( \chi_0 ) = - \chi_{1} - \chi_{-1}$ and $(1-\Phi)_1 ( \chi_{\text{rest}_1} ) = - 2 \chi_0 - \chi_{\text{rest}_2}$. It is straightforward to show that $\ker ( 1 - \Phi )_1 = \{ 0 \}$ and that every element of $\Z ( \Omega_2 )$ can be written as $a \chi_{0} + b \chi_{1} + w$ where $a,b, \in \Z$ and $w \in \im ( 1 - \Phi )_1$. Hence $\operatorname{coker} \, (1 - \Phi)_1 \cong \Z^{2}$. For $\ell \ge 2$ the maps $(1 - \Phi )_\ell : \Z ( \Omega_\ell ) \to \Z ( \Omega_{\ell+1} )$ are given by
\begin{align}
(1 - \Phi)_\ell ( \chi_n ) &= \chi_n - \chi_{n-1} - \chi_{n+1} \text{ if } n \neq \text{rest}_\ell , 0 \label{eq:second} \\
(1 - \Phi)_\ell ( \chi_0 ) &= - \chi_{1} - \chi_{-1} \label{eq:third} \\
(1 - \Phi)_\ell ( \chi_{\text{rest}_\ell} ) &= - \chi_{\ell-1}  - \chi_{-\ell+1} -  \chi_{\text{rest}_{\ell+1}}. \label{eq:first}
\end{align}
Through systematic use of \eqref{eq:first}, then \eqref{eq:second} for $n = -\ell+1 , \ldots , -1$, followed by \eqref{eq:third} and then \eqref{eq:second} for $n=\ell-1 , \ldots , 1$, one checks that  $\{ ( 1 - \Phi )_\ell ( \chi_i ) : i \in \Omega_\ell \}$ is linearly independent in $\Z ( \Omega_{\ell+1} )$, and so $\operatorname{ker} \, (1 - \Phi)_\ell = \{ 0 \}$ for $\ell \ge 2$. A similar procedure allows one to show that every element of $\Z ( \Omega_\ell )$ can be written in the form $a \chi_{0} + b \chi_{1} + w$ where $a,b, \in \Z$ and $w \in \im ( 1 - \Phi )_\ell$. Hence for $\ell \ge 2$ $\operatorname{coker} \, (1 - \Phi)_\ell = \Z ( \Omega_{\ell+1} ) / \im ( 1 - \Phi )_\ell$ is freely generated by
\begin{equation*}
\chi_0 + \im ( 1 - \Phi )_\ell \text{ and }  \chi_{1} + \im ( 1 - \Phi )_\ell .
\end{equation*}
It follows that $\operatorname{coker} \, (1 - \Phi)_\ell \cong \Z^{2}$ for all $\ell \ge 1$.

For $\ell \ge 1$ the maps $i_\ell : \Z ( \Omega_\ell ) \to \Z ( \Omega_{\ell+1} )$ are given by
\begin{align*}
i_\ell ( \chi_i ) &= \chi_i \text{ for } 0 \le \vert i \vert \le \ell -1 \text{ and } \\
i_\ell ( \chi_{\text{rest}_\ell} ) &= \chi_{\ell} + \chi_{-\ell} + \chi_{\text{rest}_{\ell+1}}
\end{align*}
and so for $\ell \ge 2$ the maps $\tilde{i}_\ell$ are the identity map. Thus, it follows from Theorem \ref{thm:howtocompute} that
\begin{equation*}
K_0 ( C^* ( E , \mathcal{L} , \mathcal{E}^{0,-} ) ) \cong \Z^2 \text{ and } K_1 ( C^* ( E , \mathcal{L} , \mathcal{E}^{0,-} ) ) \cong \{ 0 \} .
\end{equation*}
\end{example}

\begin{example}
Consider the Dyck shift $D_2$ which has labelled graph presentation $(E_2,\mathcal{L}_2)$ below:
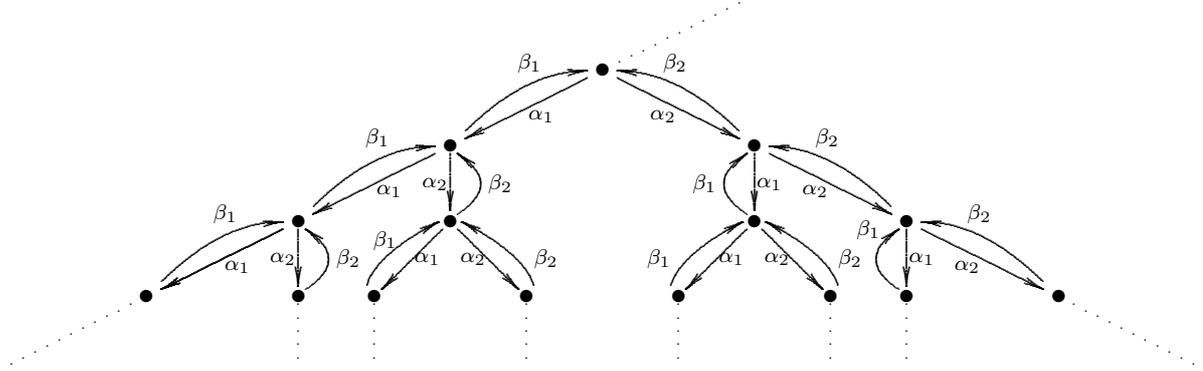
\[
\beginpicture

\setcoordinatesystem units <2cm,1cm>

\setplotarea x from -3.1 to 3, y from -2.1 to 3

\put{$\bullet$} at  0 2

\put{$\bullet$} at -1 1

\put{$\bullet$} at 1  1

\put{$\bullet$} at -1 0

\put{$\bullet$} at -2 0

\put{$\bullet$} at 1 0

\put{$\bullet$} at 2 0

\put{$\bullet$} at -1.5 -1

\put{$\bullet$} at -2 -1

\put{$\bullet$} at -0.5 -1

\put{$\bullet$} at 0.5 -1

\put{$\bullet$} at 1.5 -1

\put{$\bullet$} at 2 -1

\put{$\bullet$} at -3 -1

\put{$\bullet$} at 3 -1

\put{{\tiny $\alpha_1$}} at -0.4 1.4

\put{{\tiny $\alpha_1$}} at -1.4 0.4

\put{{\tiny $\alpha_1$}} at -2.4 -0.6

\put{{\tiny $\alpha_2$}} at 0.4 1.4

\put{{\tiny $\alpha_2$}} at 1.4 0.4

\put{{\tiny $\alpha_2$}} at 2.4 -0.6

\put{{\tiny $\beta_1$}}[r] at -0.4 2.1

\put{{\tiny $\beta_1$}}[r] at -1.4 1.1

\put{{\tiny $\beta_1$}}[r] at -2.4 0.1

\put{{\tiny $\beta_2$}}[l] at 0.4 2.1

\put{{\tiny $\beta_2$}}[l] at 1.4 1.1

\put{{\tiny $\beta_2$}}[l] at 2.4 0.1

\put{{\tiny $\alpha_2$}} at -1.1 0.5

\put{{\tiny $\beta_2$}}[l] at -0.75 0.5

\put{{\tiny $\alpha_2$}} at -2.1 -0.5

\put{{\tiny $\beta_2$}}[l] at -1.75 -0.5  %one of them

\put{{\tiny $\alpha_1$}} at 1.1 0.5

\put{{\tiny $\beta_1$}}[r] at 0.75 0.5

\put{{\tiny $\alpha_1$}} at 2.1 -0.5

\put{{\tiny $\beta_1$}}[b] at 1.75 -0.24

\put{{\tiny $\alpha_1$}} at -1.15 -0.5

\put{{\tiny $\beta_1$}}[r] at -1.35 -0.25  %another one

\put{{\tiny $\alpha_1$}} at 0.85 -0.5

\put{{\tiny $\beta_1$}}[r] at 0.45 -0.5

\put{{\tiny $\beta_2$}}[l] at -0.45 -0.5 %changed here

\put{{\tiny $\alpha_2$}} at -0.85 -0.5

\put{{\tiny $\beta_2$}}[l] at 1.55 -0.5

\put{{\tiny $\alpha_2$}} at 1.15 -0.5

\arrow <0.25cm> [0.1,0.3] from -0.1 1.9 to -0.9 1.1

\arrow <0.25cm> [0.1,0.3] from 0.1 1.9 to 0.9 1.1

\arrow <0.25cm> [0.1,0.3] from -1.1 0.9 to -1.9 0.1

\arrow <0.25cm> [0.1,0.3] from -2.1 -0.1 to -2.9 -0.9

\arrow <0.25cm> [0.1,0.3] from -2.1 -0.1 to -2.9 -0.9

\arrow <0.25cm> [0.1,0.3] from 1.1 0.9 to 1.9 0.1

\arrow <0.25cm> [0.1,0.3] from 2.1 -0.1 to 2.9 -0.9

\arrow <0.25cm> [0.1,0.3] from -1 0.9 to -1 0.15

\arrow <0.25cm> [0.1,0.3] from 1 0.9 to 1 0.15

\arrow <0.25cm> [0.1,0.3] from -2 -0.1 to -2 -0.85

\arrow <0.25cm> [0.1,0.3] from 2 -0.1 to 2 -0.85

\arrow <0.25cm> [0.1,0.3] from 1.05 -0.1 to 1.45 -0.9

\arrow <0.25cm> [0.1,0.3] from -0.95 -0.1 to -0.55 -0.9

\arrow <0.25cm> [0.1,0.3] from -1.05 -0.1 to -1.45 -0.9

\arrow <0.25cm> [0.1,0.3] from 0.95 -0.1 to 0.55 -0.9

\setquadratic

\plot -0.1 2   -0.5  1.8    -0.9 1.2 /

\plot -1.1 1   -1.5  0.8    -1.9 0.2 /

\plot -2.1 0   -2.5  -0.2    -2.9 -0.8 /

\plot 0.1 2   0.5  1.8    0.9 1.2 /

\plot 1.1 1   1.5  0.8    1.9 0.2 /

\plot 2.1 0   2.5  -0.2    2.9 -0.8 /

\plot  -0.95 0.9 -0.8 0.5  -0.95 0.1 /

\plot  0.95 0.9 0.8 0.5  0.95 0.1 /

\plot  -1.95 -0.1 -1.8 -0.5 -1.95 -0.9 /

\plot  1.95 -0.1 1.8 -0.5 1.95 -0.9 /

\plot 1.075 0 1.4 -0.4275 1.55 -0.85 /

\plot -0.925 0 -0.6 -0.4275 -0.45 -0.85 /

\plot -1.075 0 -1.4 -0.4275 -1.55 -0.85 /

\plot 0.925 0 0.6 -0.4275 0.45 -0.85 /

\arrow <0.25cm> [0.1,0.3] from -0.2 1.98 to -0.1 2

\arrow <0.25cm> [0.1,0.3] from -1.2 0.98 to -1.1 1

\arrow <0.25cm> [0.1,0.3] from -2.2 -0.02 to -2.1 0

\arrow <0.25cm> [0.1,0.3] from 0.2 1.98 to 0.1 2

\arrow <0.25cm> [0.1,0.3] from 1.2 0.98 to 1.1 1

\arrow <0.25cm> [0.1,0.3] from 2.2 -0.02 to 2.1 0

\arrow <0.25cm> [0.1,0.3] from -0.885 0.8 to -0.95 0.9

\arrow <0.25cm> [0.1,0.3] from -1.885 -0.2 to -1.95 -0.1

\arrow <0.25cm> [0.1,0.3] from 0.885 0.8 to 0.95 0.9

\arrow <0.25cm> [0.1,0.3] from 1.885 -0.2 to 1.95 -0.1

\arrow <0.25cm> [0.1,0.3] from 1.165 -0.1 to 1.075 0

\arrow <0.25cm> [0.1,0.3] from -0.835 -0.1 to -0.925 0

\arrow <0.25cm> [0.1,0.3] from 0.835 -0.1 to 0.925 0

\arrow <0.25cm> [0.1,0.3] from -1.165 -0.1 to -1.075 0

\setdots

\arrow <0cm> [0,0] from 0.9 2.9 to  0.1 2.1

\arrow <0cm> [0,0] from -3.1 -1.1 to  -3.9 -1.9

\arrow <0cm> [0,0] from -2 -1.1 to  -2 -1.9

\arrow <0cm> [0,0] from -1.5 -1.1 to -1.5 -1.9

\arrow <0cm> [0,0] from -0.5 -1.1 to -0.5 -1.9

\arrow <0cm> [0,0] from 0.5 -1.1 to 0.5 -1.9

\arrow <0cm> [0,0] from 1.5 -1.1 to 1.5 -1.9

\arrow <0cm> [0,0] from 2 -1.1 to 2 -1.9

\arrow <0cm> [0,0] from 3.1 -1.1 to 3.9 -1.9

\endpicture
\]
Since every vertex receives an edge labelled $\beta_1$, $\beta_2$, the vertices in $\Omega_\ell$ are distinguished by which labelled paths of length $\ell$ involving the symbols $\alpha_1 , \alpha_2$ they receive.

For a word $w$ in the symbols $\alpha_1 , \alpha_2$, we abuse our notation to denote the set of vertices comprising $r(w)$ by $w$. We then have
\begin{align*}
\Omega_1 &= \{ \alpha_1 , \alpha_2 \} \\
\Omega_2 &= \{ \alpha_1 \alpha_1 , \alpha_1 \alpha_2 , \alpha_2 \alpha_1 , \alpha_2 \alpha_2 \} \\
\Omega_3 &=  \{ \alpha_1 \alpha_1 \alpha_1 , \alpha_1 \alpha_1 \alpha_2 , \alpha_1 \alpha_2 \alpha_1 , \alpha_1 \alpha_2 \alpha_2 , \alpha_2 \alpha_1 \alpha_1 , \alpha_2 \alpha_1 \alpha_2 , \alpha_2 \alpha_2 \alpha_1 , \alpha_2 \alpha_2 \alpha_2 \} \\
\vdots & \\
\Omega_\ell &= \{ \alpha_1^\ell , \alpha_1^{\ell-1} \alpha_2 , \alpha_1^{\ell-2} \alpha_2 \alpha_1 , \alpha_1^{\ell-2} \alpha_2 \alpha_2 , \ldots , \alpha_2^\ell \} .
\end{align*}
For $\ell=1$ we have
\begin{align*}
(1 - \Phi)_1 ( \chi_{\alpha_1} ) &= - \chi_{\alpha_1 \alpha_1} - 2 \chi_{\alpha_1 \alpha_2} - \chi_{\alpha_2 \alpha_2} \\
(1 - \Phi)_1 ( \chi_{\alpha_2} ) &=  - \chi_{\alpha_1 \alpha_1} - 2 \chi_{\alpha_2 \alpha_1} - \chi_{\alpha_2 \alpha_2}.
\end{align*}
It is straightforward to see that $\operatorname{ker} \, (1 - \Phi)_1 = \{0 \}$. Let
\begin{equation*}
x_{11} =  \chi_{\alpha_1 \alpha_1} + \im (1 - \Phi)_1 , x_{12} =  \chi_{\alpha_1 \alpha_2} + \im (1 - \Phi)_1 , x_{21} =  \chi_{\alpha_2 \alpha_1}  + \im (1 - \Phi)_1 , x_{22} =  \chi_{\alpha_2 \alpha_2} + \im (1 - \Phi)_1.
\end{equation*}
Then $2(x_{11}+x_{12}+x_{21}+x_{22})=0$ since
\begin{equation*}
(1 - \Phi )_1 ( \chi_{\alpha_1} + \chi_{\alpha_2} ) = 2 ( \chi_{\alpha_1 \alpha_1} + \chi_{\alpha_1 \alpha_2} + \chi_{\alpha_2 \alpha_1}  + \chi_{\alpha_2 \alpha_2}).
\end{equation*}
It follows that $\operatorname{coker} \, (1 - \Phi)_1 \cong \Z^2 \times \left(\Z / 2 \Z \right)$, generated by $x_{12}, x_{22}, x_{11}+x_{12}+x_{21}+x_{22}$.

As every vertex emits an edge labelled $\alpha_1$ and an edge labelled $\alpha_2$, we have $r ( w , \alpha_i ) = w \alpha_i$ for $i=1,2$ and all words $w$ in the symbols $\alpha_1 , \alpha_2$. Furthermore $r ( w \alpha_i , \beta_j ) = w$ if $i=j$ and is empty if $i \neq j$. Since $[w]_\ell = [ \alpha_1 w ]_{\ell+1} \cup [ \alpha_2 w ]_{\ell+1}$ for all words $w\in\Omega_\ell$, we have $i_\ell ( \chi_w ) = \chi_{\alpha_1 w} + \chi_{\alpha_2 w}$, and so
\begin{align*}
(1 - \Phi)_2 ( \chi_{\alpha_1 \alpha_1} ) &= - \chi_{\alpha_1 \alpha_1 \alpha_1} - \chi_{\alpha_1 \alpha_1 \alpha_2} - \chi_{\alpha_1 \alpha_2 \alpha_1} - \chi_{\alpha_2 \alpha_2 \alpha_1}  \\
(1 - \Phi)_2 ( \chi_{\alpha_1 \alpha_2} ) &= - \chi_{\alpha_1 \alpha_1 \alpha_1} + \chi_{\alpha_1 \alpha_1 \alpha_2} -2  \chi_{\alpha_1 \alpha_2 \alpha_1} - \chi_{\alpha_1 \alpha_2 \alpha_2} -  \chi_{\alpha_2 \alpha_1 \alpha_1} + \chi_{\alpha_2 \alpha_1 \alpha_2} - \chi_{\alpha_2 \alpha_2 \alpha_1}    \\
( 1 - \Phi)_2 ( \chi_{\alpha_2 \alpha_1} ) &= - \chi_{\alpha_1 \alpha_1 \alpha_2} + \chi_{\alpha_1 \alpha_2 \alpha_1} -  \chi_{\alpha_1 \alpha_2 \alpha_2}   - \chi_{\alpha_2 \alpha_1 \alpha_1} - 2 \chi_{\alpha_2 \alpha_1 \alpha_2} + \chi_{\alpha_2 \alpha_2 \alpha_1}
  - \chi_{\alpha_2 \alpha_2 \alpha_2} \\
(1 - \Phi)_2 ( \chi_{\alpha_2 \alpha_2} ) &= - \chi_{\alpha_1 \alpha_1 \alpha_2} - \chi_{\alpha_2 \alpha_1 \alpha_2} - \chi_{\alpha_2 \alpha_2 \alpha_1}
- \chi_{\alpha_2 \alpha_2 \alpha_2} .
\end{align*}
One checks that $\operatorname{ker} \, (1 - \Phi)_2 = \{0 \}$. For a word $w\in\Omega_3$, let $x_w = \chi_w + \im (1 - \Phi)_2$. Then $2\sum_{w\in\Omega_3} x_w = 0$, because
\begin{equation*}
(1 - \Phi )_2 ( \chi_{\alpha_1 \alpha_2} + \chi_{\alpha_2 \alpha_2} + \chi_{\alpha_1 \alpha_1} + \chi_{\alpha_2 \alpha_2}  ) = 2\sum_{w\in\Omega_3} \chi_{w}.
\end{equation*}
It follows that  $\operatorname{coker} \, (1 - \Phi)_2 \cong \Z^4 \times \left( \Z /  2 \Z \right)$, generated by $x_{112}, x_{122}, x_{212}, x_{222}, \sum_{w\in\Omega_3} x_w$.

Since $\tilde{i}_2 ( x_{12} ) = x_{112} + x_{212}$, $\tilde{i}_2 ( x_{22} ) = x_{122} + x_{222}$ and $\tilde{i}_2 ( x_{11} + x_{12} + x_{21} + x_{22} ) = \sum_{w\in\Omega_3} x_w$ we see that $\tilde{i}_2 : \Z^2 \times \left( \Z / 2 \Z \right) \to \Z^4 \times \left( \Z /  2 \Z \right)$ is the identity on the cyclic group of order $2$ and an injection on the free abelian part.

In general we have for a word $w$ of length $n$ in $\Omega_n$ that
\begin{equation*}
(1 - \Phi)_n ( \chi_w ) = \chi_{\alpha_1 w} + \chi_{\alpha_2 w} - \chi_{w \alpha_1} - \chi_{w \alpha_2} - \chi_{\alpha_1 \alpha_1 w_{n-1}}
- \chi_{\alpha_2 \alpha_1 w_{n-1}}- \chi_{\alpha_1 \alpha_2 w_{n-1}}- \chi_{\alpha_2 \alpha_2 w_{n-1}}
\end{equation*}
where $w_{n-1}$ represents the first $n-1$ symbols of $w$. Again, a short calculation
shows that $\operatorname{ker} \, (1 - \Phi)_n = \{ 0 \}$. For a word $w\in\Omega_{n+1}$, let $x_w = \chi_w + \im (1 - \Phi)_n$. Then $2\sum_{w\in\Omega_{n+1}} x_w = 0$, because
\begin{equation*}
(1 - \Phi )_n \Bigg( \sum_{w'\in\Omega_n} \chi_{w'} \Bigg) = 2\sum_{w\in\Omega_{n+1}} \chi_{w}.
\end{equation*}
Using this, one can show that $\operatorname{coker} \, (1 - \Phi)_n \cong \Z^{2^n} \times \left( \Z /  2 \Z \right)$, generated by $\{x_{w'2} : w'\in\Omega_n\}\cup\{\sum_{w\in\Omega_{n+1}} x_w\}$.

Observe that $\tilde{i}_n$ is induced by the map $i_n$ which only adds symbols to the beginning of words. It follows that $\tilde{i}_n: \Z^{2^n} \times \left( \Z /  2 \Z \right) \to \Z^{2^{n+1}} \times \left( \Z /  2 \Z \right)$ is the identity on the cyclic group of order $2$ and an injection on the free abelian part. From Theorem \ref{thm:howtocompute} we have that
\[
K_0 ( C^* ( E_2 , \mathcal{L}_2 , \mathcal{E}_2^{0,-} ) ) \cong
\left( \coprod_{i=1}^\infty \Z \right) \times \left( {\Z} / 2 {\Z} \right)
\text{ and } K_1 ( C^* ( E_2 , {\mathcal L}_2 , {\mathcal E}_2^{0,-}) ) \cong 0 .
\]
This confirms the calculations done in \cite[\S 3]{km}. It is worth
noting that $C^* ( E_2 , \mathcal{L}_2 , \mathcal{E}_2^{0,-} ) $ is
unital (the unit is $p_{r(\alpha_1)}+p_{r(\alpha_2)}$) and has a $K_0$-group which is not finitely generated, and so cannot be isomorphic to a graph algebra.
\end{example}

\begin{example}
Consider the following left-resolving labelled graph $( E, \labe )$ over the infinite alphabet $\{0,1\} \times \Z$:
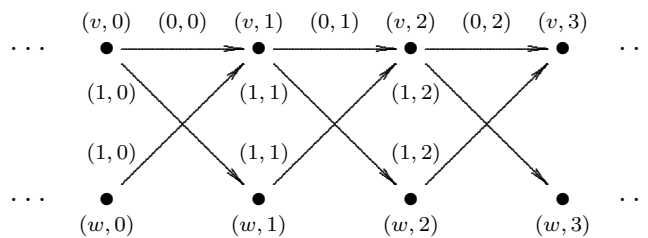
\begin{equation*}
\beginpicture

\setcoordinatesystem units <2cm,2cm>

\setplotarea x from -2 to 3, y from 0 to 1

\put{$\bullet$} at -1 1

\put{\tiny $(v,0)$}[b] at -1 1.1

\put{$\bullet$} at 0 1

\put{\tiny $(v,1)$}[b] at 0 1.1

\put{$\bullet$} at 1 1

\put{\tiny $(v,2)$}[b] at 1 1.1

\put{$\bullet$} at 2 1

\put{\tiny $(v,3)$}[b] at 2 1.1

\put{$\bullet$} at -1 0

\put{\tiny $(w,0)$}[t] at -1 -0.1

\put{$\bullet$} at 0 0

\put{\tiny $(w,1)$}[t] at 0 -0.1

\put{$\bullet$} at 1 0

\put{\tiny $(w,2)$}[t] at 1 -0.1

\put{$\bullet$} at 2 0

\put{\tiny $(w,3)$}[t] at 2 -0.1

\put{$\dots$} at -1.5 1

\put{$\dots$} at 2.5 1

\put{$\dots$} at -1.5 0

\put{$\dots$} at 2.5 0

\arrow <0.25cm> [0.1,0.3] from -0.9 1 to -0.1 1

\put{\tiny $(0,0)$}[b] at -0.5 1.1

\arrow <0.25cm> [0.1,0.3] from 0.1 1 to 0.9 1

\put{\tiny $(0,1)$}[b] at 0.5 1.1

\arrow <0.25cm> [0.1,0.3] from 1.1 1 to 1.9 1

\put{\tiny $(0,2)$}[b] at 1.5 1.1

\arrow <0.25cm> [0.1,0.3] from -0.9 0.9 to -0.1 0.1

\put{\tiny $(1,0)$}[r] at -0.8 0.7

\arrow <0.25cm> [0.1,0.3] from 0.1 0.9 to 0.9 0.1

\put{\tiny $(1,1)$}[r] at 0.2 0.7

\arrow <0.25cm> [0.1,0.3] from 1.1 0.9 to 1.9 0.1

\put{\tiny $(1,2)$}[r] at 1.2 0.7

\arrow <0.25cm> [0.1,0.3] from -0.9 0.1 to -0.1 0.9

\put{\tiny $(1,0)$}[r] at -0.8 0.3

\arrow <0.25cm> [0.1,0.3] from 0.1 0.1 to 0.9 0.9

\put{\tiny $(1,1)$}[r] at 0.2 0.3

\arrow <0.25cm> [0.1,0.3] from 1.1 0.1 to 1.9 0.9

\put{\tiny $(1,2)$}[r] at 1.2 0.3

\endpicture
\end{equation*}
We write $E^0 = \{ ( v, i ) , ( w , i ) : i \in \Z \}$, and label the edges as shown. Since $(v,i)$ is the only vertex which receives an edge labelled $(0,i-1)$, we have $[(v,i) ]_\ell = \{ (v,i) \}$ for all $i \in \Z$ and $\ell \ge 1$. Since $(w,i)$ is the only vertex which receives an edge with label $(1,i-1)$, but does not receive any edge with label $(0,j)$ for any $j \in \Z$, we have $[(w,i) ]_\ell = \{ (w,i) \}$ for all $i \in \Z$ and $\ell \ge 1$. We will simply identify $[(v,i) ]_\ell = \{ (v,i) \}$ with $(v,i)$ and $[(w,i) ]_\ell = \{ (w,i) \}$ with $(w,i)$ for all $i \in \Z$. We then have that $\Omega_\ell = E^0$ for all $\ell \ge 1$.

Since $\Omega_\ell = E^0$ for each $\ell \ge 1$, and $E^0$ is infinite, we have
\begin{equation*}
\Z ( \Omega_\ell ) = \bigoplus_{E^0} \Z = \spa_{\Z} \{ \chi_{(v,i)} , \chi_{(w,i)} : i \in \Z \} .
\end{equation*}
Since $(v,i)$ only connects to $(v,i+1)$ and $(w,i+1)$ and $(w,i)$ only connects to $(v,i+1)$ for all $i \in \Z$, we have
\begin{align}
(1 - \Phi)_\ell ( \chi_{(v,i)} ) &= \chi_{(v,i)} - \chi_{(v,i+1)} - \chi_{(w,i+1)} \text{ and} \label{eq:one} \\
(1 - \Phi)_\ell ( \chi_{(w,i)} ) &= \chi_{(w,i)} - \chi_{(v,i+1)} .  \label{eq:two}
\end{align}
We claim that $\operatorname{ker} \, (1 - \Phi)_\ell = \{ 0 \}$ for all $\ell \ge 1$.

Suppose, for contradiction, that $b  \in \operatorname{ker} ( (1 - \Phi)_\ell ) \subseteq \oplus_{E^0}  \Z$ is a nonzero vector. Since $b = \sum_{i \in \mathbf{Z}} b_{(v,i)} \chi_{(v,i)} + b_{(w,i)} \chi_{(w,i)}$ is in the direct sum there must be a maximum $i$ such that not both $b_{(v,i)}$ and $b_{(w,i)}$ are 0. Then by \eqref{eq:one} we have $b_{(v,i)}=b_{(v,i+1)}+b_{(w,i+1)}=0$, and by \eqref{eq:two} we have $b_{(w,i)}=b_{(v,i+1)}=0$ which is a contradiction. Hence $\operatorname{ker} \, (1 - \Phi)_\ell = \{ 0 \}$ for all $\ell \ge 1$.

Now we compute the cokernel of $( 1 - \Phi)_\ell$. From Equations \eqref{eq:one} and \eqref{eq:two} we see that
\begin{equation} \label{eq:gettingridofvs}
( 1 - \Phi)_\ell ( \chi_{(v,i)} - \chi_{(w,i)} ) = \chi_{(v,i)} - \chi_{(w,i)} - \chi_{(w,i+1)} .
\end{equation}
Repeated use of Equation \eqref{eq:gettingridofvs} shows that every vector in $\Z ( \Omega_\ell)$ is equivalent, using vectors in $\operatorname{Im} \, (1 - \Phi)_\ell$, to a vector which has all $(v,i)$--coordinates zero. From Equations \eqref{eq:one} and \eqref{eq:two} we also see that
\begin{equation} \label{eq:gettingridofws}
(1 - \Phi)_\ell ( \chi_{(v,i)} + \chi_{(w,i-1)} - \chi_{(w,i)} ) = \chi_{(w,i-1)} - \chi_{(w,i)} - \chi_{(w,i+1)}.
\end{equation}
Repeated use of Equation \eqref{eq:gettingridofws} shows that every vector in $\Z ( \Omega_\ell)$ is equivalent, using vectors in $\operatorname{Im} \, (1 - \Phi)_\ell$, to a vector which has all $(v,i)$-coordinates zero and all $(w,i)$-coordinates zero except for $(w,0)$ and $(w,1)$, with no relation between them. Hence $\operatorname{coker} \, (1 - \Phi)_\ell$ is generated by  $\chi_{(w,0)} + \im (1 - \Phi)_\ell$ and $ \chi_{(w,1)} + \im (1 - \Phi)_\ell$ for all $\ell \ge 1$ and so we may conclude that $\operatorname{coker} \, (1 - \Phi)_\ell \cong \Z^2$ for all $\ell \ge 1$.

Since $\tilde{i}_\ell$ is the identity map for all $\ell \ge 1$, we have
\begin{equation*}
K_0 ( C^* ( E , \mathcal{L} , \mathcal{E}^{0,-} ) ) \cong \mathbf{Z}^2 \text{ and } K_1 ( C^* ( E , \mathcal{L} , \mathcal{E}^{0,-} ) ) \cong \{ 0 \} .
\end{equation*}
\end{example}

\appendix

\section{$C^*$-algebras of non-normal labelled spaces}

In this appendix we give an example showing that if $(E,\mathcal{L},\mathcal{B})$ is a weakly left-resolving labelled space which is not normal, and we define a representation of $(E,\mathcal{L},\mathcal{B})$ as in Definition \ref{lgdef} and $C^*(E,\mathcal{L},\mathcal{B})$ as in Definition \ref{def:celb}, then $C^*(E,\mathcal{L},\mathcal{B})$ might not satisfy the gauge invariant uniqueness theorem (Corollary \ref{cor:giut}).  We will then describe how to modify the definition of a representation of a non-normal weakly left-resolving labelled space, and consequently of $C^*(E,\mathcal{L},\mathcal{B})$, in order for $C^*(E,\mathcal{L},\mathcal{B})$ to satisfy the gauge invariant uniqueness theorem as well as Theorem \ref{theorem:main}, Corollary \ref{corol:iota}, and Corollary \ref{cor:nuclear}.

\begin{example}
Let $E$ be the graph with $E^0=\{v_0,v_1\}$, $E^1=\{e_n:n\in\N\}$, $r(e_0)=s(e_0)=v_0$, and $r(e_n)=s(e_n)=v_1$ for $n>0$. Define a labelling map $\mathcal{L}:E^1\to\{a_n:n\in\{1,2,\dots\}\}$ by $\mathcal{L}(e_0)=a_1$ and $\mathcal{L}(e_n)=a_n$ for $n>0$.
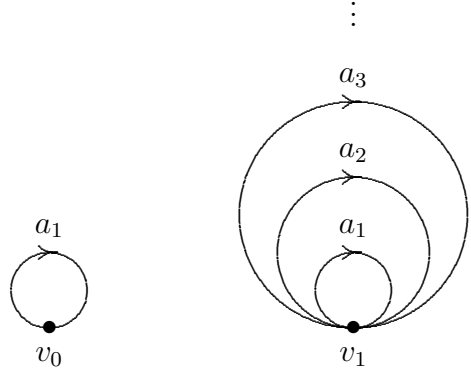
\begin{equation*}
\beginpicture
\setcoordinatesystem units <2cm,1cm>

\setplotarea x from -3.1 to 3, y from  0.3  to 1

\put{$\bullet$} at -1 0

\put{$v_0$}[b] at -1 -0.5

\put{$\bullet$} at 1 0

\put{$v_1$}[b] at 1 -0.5

\put{$a_1$}[b] at -1 1.2

\circulararc 360 degrees from -1 0 center at -1 0.5

\arrow <0.25cm> [0.1,0.8] from -1 1 to -0.95 1

\put{$a_1$}[b] at 1 1.2

\circulararc 360 degrees from 1 0 center at 1 0.5

\arrow <0.25cm> [0.1,0.8] from 1 1 to 1.05 1

\put{$a_2$}[b] at 1 2.2

\circulararc 360 degrees from 1 0 center at 1 1

\arrow <0.25cm> [0.1,0.8] from 1 2 to 1.05 2

\put{$a_3$}[b] at 1 3.2

\circulararc 360 degrees from 1 0 center at 1 1.5

\arrow <0.25cm> [0.1,0.8] from 1 3 to 1.05 3

\put{$\vdots$}[b] at 1 4

\endpicture
\end{equation*}
Then $(E,\mathcal{L})$ is a left-resolving labelled graph. Let $\bool=\{E^0,\{v_1\}\}$. Then $(E,\mathcal{L},\mathcal{B})$ is a weakly left-resolving labelled space, but it is not normal.

We will now construct two representations of $(E,\mathcal{L},\mathcal{B})$. Let $H_1$ be a Hilbert space with an orthonormal basis $(e_{(x,n)})_{(x,n)\in (\N\cup \N^\N)\times\Z}$ indexed by $(\N\cup \N^\N)\times\Z$ (we can, for example, let $H_1$ be $l^2((\N\cup \N^\N)\times\Z)$) and define bounded operators $q_{E^0}, q_{\{v_1\}}, t_{a_1}, t_{a_2},\dots$ on $H_1$ in the following way:
\begin{equation*}
	q_{E^0}e_{(x,n)}=e_{(x,n)}\text{ for all }x\in\N\cup \N^\N,\quad q_{\{v_1\}}e_{(x,n)}=
	\begin{cases}
		0&\text{if }x\in\N,\\
		e_{(x,n)}&\text{if }x\in\N^\N,
	\end{cases}
\end{equation*}
\begin{equation*}
	t_{a_1}e_{(x,n)}=
	\begin{cases}
		e_{(x+1,n)}&\text{if }x\in\N,\\
		e_{(1x,n)}&\text{if }x\in\N^\N,
	\end{cases}
	\quad
	t_{a_i}e_{(x,n)}=
	\begin{cases}
		0&\text{if }x\in\N,\\
		e_{(ix,n)}&\text{if }x\in\N^\N,
	\end{cases}
	\text{ for }i\ge 2.
\end{equation*}
Then $q_{E^0}$ and $q_{\{v_0\}}$ are projections and $ t_{a_1}, t_{a_2},\dots$ are partial isometries satisfying
\begin{itemize}

\item[(i)] If $A, B \in \mathcal{B} $, then $q_A q_B = q_{A \cap B}$ and $q_{A \cup B} = q_A + q_B - q_{A \cap B}$, where $q_\emptyset = 0$.

\item[(ii)] If $a \in \mathcal{A}$ and $A \in \mathcal{B} $, then $q_A t_a = t_a q_{r ( A, a )}$.

\item[(iii)] If $a , b \in \mathcal{A}$, then $t_a^* t_a = p_{r( a )}$ and $t_{a}^*t_{b} = 0$ unless $a = b$.

\item[(iv)] For $A \in \mathcal{B} $  with $\mathcal{L}(A E^1)$ finite and $A \cap B = \emptyset$ for all $B\in\mathcal{B}$ satisfying $B\subseteq\sink$, we have
\begin{equation*}
q_A = \sum_{a \in \mathcal{L} ( A E^1 )} t_{a} q_{r( A , a )} t_{a}^*
\end{equation*}
\end{itemize}
(condition (iv) is satisfied because there are no $A \in \mathcal{B} $  with $\mathcal{L}(A E^1)$ finite). For each $z\in\mathcal{T}$ there is an automorphism $\beta_z$ on $B(H_1)$ given by $\beta_z(x)=u_zxu_z^*$ for $x\in B(H_1)$ where $u_z$ is the unitary operator given by $u_ze_{(x,n)}=e_{(x,n+1)}$ for $(x,n)\in (\N\cup\N^\N)\times\Z$. Then $\beta_z(q_A)=q_A$ for all $A\in\bool$, and $\beta_z(t_a)=zt_a$ for all $a\in\mathcal{A}$.

Similarly, let $H_2$ be a Hilbert space with an orthonormal basis $(f_y)_{y\in\Z\cup\{\N\}^\N}$ indexed by $(\Z\cup \N^\N)\times\Z$ and define bounded operators $r_{E^0}, r_{\{v_1\}}, u_{a_1}, u_{a_2},\dots$ on $H_2$ in the following way:
\begin{equation*}
	r_{E^0}f_{(y,n)}=f_{(y,n)}\text{ for all }y\in\Z\cup \N^\N,\quad r_{\{v_1\}}f_{(y,n)}=
	\begin{cases}
		0&\text{if }y\in\Z,\\
		f_{(y,n)}&\text{if }y\in\N^\N,
	\end{cases}
\end{equation*}
\begin{equation*}
	u_{a_1}f_{(y,n)}=
	\begin{cases}
		f_{(y+1,n)}&\text{if }y\in\Z,\\
		f_{(1y,n)}&\text{if }y\in\N^\N,
	\end{cases}
	\quad
	u_{a_i}f_{(y,n)}=
	\begin{cases}
		0&\text{if }y\in\Z,\\
		f_{(iy,n)}&\text{if }y\in\N^\N,
	\end{cases}
	\text{ for }i\ge 2.
\end{equation*}
Then $r_{E^0}$ and $r_{\{v_0\}}$ are projections and $ u_{a_1}, u_{a_2},\dots$ are partial isometries satisfying
\begin{itemize}

\item[(i)] If $A, B \in \mathcal{B} $, then $r_A r_B = r_{A \cap B}$ and $r_{A \cup B} = r_A + r_B - r_{A \cap B}$, where $r_\emptyset = 0$.

\item[(ii)] If $a \in \mathcal{A}$ and $A \in \mathcal{B} $, then $r_A u_a = u_a r_{r ( A, a )}$.

\item[(iii)] If $a , b \in \mathcal{A}$, then $u_a^* u_a = r_{r( a )}$ and $u_{a}^*u_{b} = 0$ unless $a = b$.

\item[(iv)] For $A \in \mathcal{B} $  with $\mathcal{L}(A E^1)$ finite and $A \cap B = \emptyset$ for all $B\in\mathcal{B}$ satisfying $B\subseteq\sink$, we have
\begin{equation*}
r_A = \sum_{a \in \mathcal{L} ( A E^1 )} u_{a} r_{r( A , a )} u_{a}^*
\end{equation*}
\end{itemize}
For each $z\in\bf{T}$ there is an automorphism $\gamma_z$ on $B(H_2)$ given by $\gamma_z(x)=v_zxv_z^*$ for $x\in B(H_2)$ where $v_z$ is the unitary operator given by $v_ze_{(y,n)}=e_{(y,n+1)}$ for $(y,n)\in (\Z\cup\N^\N)\times\Z$. Then $\gamma_z(r_A)=r_A$ for all $A\in\bool$, and $\gamma_z(u_a)=zu_a$ for all $a\in\mathcal{A}$.

Suppose we define a representation of $(E,\mathcal{L},\mathcal{B})$ as in Definition \ref{lgdef} and let $C^*(E,\mathcal{L},\mathcal{B})$  be the $C^*$-algebra generated by a universal representation of $(E,\mathcal{L},\mathcal{B})$. Then $C^*(E,\mathcal{L},\mathcal{B})$ cannot satisfy a gauge invariant uniqueness theorem because that would imply that there is a $*$-isomorphism from the $C^*$-algebra generated by $\{q_{E^0},q_{\{v_1\}}, t_{a_1}, t_{a_2},\dots\}$ to the $C^*$-algebra generated by $\{r_{E^0},r_{\{v_1\}}, u_{a_1}, u_{a_2},\dots\}$ mapping $q_{E^0}$ to $r_{E^0}$, $q_{\{v_1\}}$ to $r_{\{v_1\}}$, and $t_{a_i}$ to $u_{a_i}$ for each $i$.  This cannot be the case since $q_{E^0}-q_{\{v_1\}}\ne (q_{E^0}-q_{\{v_1\}})t_{a_1}t_{a_1}^*$ whereas $r_{E^0}-r_{\{v_1\}}= (r_{E^0}-r_{\{v_1\}})u_{a_1}u_{a_1}^*$.
\end{example}

\begin{remark} \label{remark:nnd}
Let $(E,\mathcal{L},\mathcal{B})$ be a weakly left-resolving labelled space. Even if $(E,\mathcal{L},\mathcal{B})$ is not normal the $C^*$-algebra $\Aalg$ and the $C^*$-correspondence $(X(E,{\mathcal L},{\mathcal B}),\varphi)$ can still be defined as in Section \ref{sec:lga=cpa}. Let $$\mathcal{P}=\{(A,B): A\in\mathcal{B},\ B\in\mathcal{B}\cup\emptyset,\ B\subseteq A\}$$ and suppose $(A,B)\in\mathcal{P}$. Notice that $r(B,a)\subseteq r(A,a)$ for all $a\in\mathcal{A}$ (because $B\subseteq A$). Let $$\mathcal{A}_{(A,B)}=\{a\in\mathcal{A}: r(B,a)\ne r(A,a)\}$$ and let $$\mathcal{N}=\{A\setminus B: (A,B)\in\mathcal{P},\ \mathcal{A}_{(A,B)}=\emptyset\}.$$ An analysis of $(X(E,{\mathcal L},{\mathcal B}),\varphi)$ similar to the one given in Section \ref{sec:lga=cpa} shows that $\Oalg_{(X(E,{\mathcal L},{\mathcal B}),\varphi)}$ is the universal $C^*$-algebra generated by a family $\{ p_A : A \in \bool \}$ of projections and a family  $\{ s_a : a \in \mathcal{A} \}$ of partial isometries satisfying
\begin{itemize}

\item[(i')] If $A, B \in \bool$, then $p_A p_B = p_{A \cap B}$ and $p_{A \cup B} = p_A + p_B - p_{A \cap B}$, where $p_\emptyset = 0$.

\item[(ii')] If $a \in \mathcal{A}$ and $A \in \mathcal{B} $, then $p_A s_a = s_a p_{r ( A, a )}$.

\item[(iii')] If $a , b \in \mathcal{A}$, then $s_a^* s_a = p_{r( a )}$ and $s_{a}^*s_{b} = 0$ unless $a = b$.

\item[(iv')] For $(A,B)\in\mathcal{P}$ with $\mathcal{A}_{(A,B)}$ finite and $A\cap C=B\cap C$ for all $C\in\mathcal{N}$ we have
\begin{equation*}
p_A =p_B + \sum_{a \in \mathcal{A}_{(A,B)}} s_{a} (p_{r( A , a )}-p_{r(B,a)}) s_{a}^* .
\end{equation*}
\end{itemize}
\end{remark}

We therefore define a representation of $(E,\mathcal{L},\mathcal{B})$ and $C^*(E,\mathcal{L},\mathcal{B})$ in the following way.

\begin{definition} \label{mlgdef}
Let $( E , \mathcal{L}  , \mathcal{B} )$ be a weakly left-resolving  labelled space. A \emph{representation} of $( E , \mathcal{L} , \mathcal{B} )$ in a $C^*$-algebra consists of projections $\{ p_A : A \in \mathcal{B} \}$ and partial isometries $\{ s_a : a \in \mathcal{A} \}$ satisfying the conditions (i')--(iv') above.
\end{definition}

\begin{definition} \label{def:mcelb}
Let $( E , \mathcal{L}  , \mathcal{B}  )$ be a  weakly left-resolving  labelled space. Then $C^* ( E , \mathcal{L}  , \mathcal{B}  )$ is the $C^*$-algebra generated by a universal representation of $( E , \mathcal{L} , \mathcal{B}  )$.
\end{definition}

\begin{remark}
	If $( E , \mathcal{L}  , \mathcal{B}  )$ is a  weakly left-resolving normal labelled space, then the definition of a representation of $( E , \mathcal{L}  , \mathcal{B}  )$ given in Definition \ref{mlgdef} agrees with the definition of a representation of $( E , \mathcal{L}  , \mathcal{B}  )$ given in Definition \ref{lgdef}.
\end{remark}

\begin{remark}
	If $( E , \mathcal{L}  , \mathcal{B}  )$ is a  weakly left-resolving  labelled space such that $r(A\setminus B,a)=r(A,a)\setminus r(B,a)$ for all $(A,B)\in\mathcal{P}$ and all $a\in\mathcal{A}$ (this will always be the case if $(E,\mathcal{L})$ is left-resolving), and we let $\tbool$ be the collection of subsets of $E^0$ which can be written as a finite disjoint union of sets of the form $A\setminus B$ where $(A,B)\in\mathcal{P}$, then $( E , \mathcal{L}  , \tbool  )$ is a  weakly left-resolving normal labelled space and $C^* ( E , \mathcal{L}  , \mathcal{B}  ) \cong C^* ( E , \mathcal{L}  , \tbool  )$.
\end{remark}

\begin{remark} \label{remark:nn}
It follows from Remark \ref{remark:nnd} that with the definition of a representation of a weakly left-resolving labelled space $( E , \mathcal{L}  , \mathcal{B}  )$ given in Definition \ref{mlgdef} and the definition of the $C^*$-algebra $C^* ( E , \mathcal{L}  , \mathcal{B}  )$ given in Definition \ref{def:mcelb} that Theorem \ref{theorem:main}, and therefore also Corollary \ref{corol:iota} and Corollary \ref{cor:nuclear}, remain true if the assumption that $( E , \mathcal{L}  , \mathcal{B}  )$ is normal is dropped.
\end{remark}

We also get the following gauge-invariant uniqueness Theorem for $C^* ( E , \mathcal{L}  , \mathcal{B}  )$.

\begin{corollary}  \label{cor:giutm}
Let $(E,\labe,\bool)$ be a weakly left-resolving labelled space. Let $\{ p_A,\ s_a : A \in \mathcal{B},\ a \in \mathcal{A} \}$ be the universal representation of $( E , \mathcal{L}  , \mathcal{B}  )$ that generates $C^* ( E , \mathcal{L}  , \mathcal{B}  )$, let $\{ q_A,\ t_a : A \in \mathcal{B},\ a \in \mathcal{A} \}$ be a representation of $( E , \mathcal{L}  , \mathcal{B}  )$ in a $C^*$-algebra $\Xalg$ and let $\pi$ be the unique $*$-homomorphism from $C^* ( E , \mathcal{L}  , \mathcal{B}  )$ to $\Xalg$ that maps each $p_A$ to $q_A$ and each $s_a$ to $t_a$. Then $\pi$ is injective if and only if $q_A=q_B$ implies that $A=B$ for all $(A,B)\in\mathcal{P}$, and  for each $z\in\mathbb{T}$ there exists a $*$-homomorphism $\beta_z:C^*(\{ q_A,\ t_a : A \in \mathcal{B},\ a \in \mathcal{A} \})\to C^*(\{ q_A,\ t_a : A \in \mathcal{B},\ a \in \mathcal{A} \})$ such that $\beta_z(q_A)=q_A$ and $\beta_z(t_a)=zt_a$ for $A\in\bool$ and $a\in\mathcal{A}$.
\end{corollary}

The $K$-theory of $C^* ( E , \mathcal{L}  , \mathcal{B}  )$ can be determined by the following adaptation of Theorem \ref{theorem:ktheory}.

\begin{theorem}\label{thm:ktheorym}
Let $(E,\labe,\bool)$ be a weakly left-resolving labelled space. Let $$\mathcal{P}_J=\{(A,B)\in\mathcal{P}: \mathcal{A}_{(A,B)}\text{ is finite, and }A\cap C=B\cap C\text{ for all }C\in\mathcal{N}\},$$
let $$\boolj=\{A\setminus B:(A,B)\in\mathcal{P}_J\},$$ and let $(1-\Phi)$ be the linear map from $\spa_\Z\{\cha{A}: A\in\boolj\}$ to $\spa_\Z \{ \cha{A} :  A \in \hbool \}$ given by
\begin{equation*}
(1-\Phi) ( \cha{A\setminus B} ) = \cha{A\setminus B} - \sum_{a \in \mathcal{A}_{(A,B)}} \cha{r(A,a)\setminus r(B,a)}, \quad  (A,B)\in\mathcal{P}_J.
\end{equation*}
Then $K_1(C^*(E,\labe,\bool)) \cong \ker(1-\Phi)$, and $K_0(C^*(E,\labe,\bool)) \cong \spa_\Z\{\cha{A}: A\in\hbool\}/\im(1-\Phi)$ via $[p_A]_0 \mapsto \cha{A}+\im(1-\Phi)$ for $A \in \hbool$.
\end{theorem}

If $(E,\labe,\bool)$ is normal, then the definitions of $\boolj$ and $(1-\Phi)$ given in the theorem above agree with the definition of $\boolj$ given by Equation \eqref{eq:boolj} and the definition of $(1-\Phi)$ given in Theorem \ref{theorem:ktheory}, so the theorem above reduces to Theorem \ref{theorem:ktheory} in this case.

\end{document}